\documentclass[12pt]{article}

\usepackage{amsmath, amsfonts, amscd, latexsym, amsthm, amssymb, algorithm, algpseudocode, titlesec, array}
\usepackage[hang,flushmargin]{footmisc} 
\usepackage{graphicx}
\usepackage{cancel}
\usepackage{multirow}
\usepackage{tikz} 
\usetikzlibrary{arrows}
\usepackage{xcolor}
\usepackage[margin=1in]{geometry}
\usepackage{appendix} 

\theoremstyle{plain}
\newtheorem{Th}{Theorem}[section]
\newtheorem{Lem}{Lemma}[section]
\newtheorem{Prop}{Proposition}[section]
\newtheorem{Cor}{Corollary}[section]

\newcommand{\tee}{\tau(A,u)}
\newcommand{\be}{\bar e}
\newcommand{\pa}{P_{A}}
\newcommand{\adj}{^\top}
\newcommand{\cc}{\rho(A)}
\newcommand{\bv}{\bar v}
\newcommand{\cp}{{\cal P}}
\newcommand{\hlambda}{\hat\lambda}
\newcommand{\tlambda}{\tilde\lambda}
\newcommand{\blambda}{\bar\lambda}
\newcommand{\clambda}{\check\lambda}
\newcommand{\poi}{(P)}
\newcommand{\poip}{(P')}
\newcommand{\alt}{\mathrm{({\it Alt})}}
\newcommand{\lbi}{(\mathit{LB}_i)} 
\newcommand{\lb}{(\mathit{LB})}
\newcommand{\1}{L}
\newcommand{\2}{Q}
\newcommand{\3}{E}
\newcommand{\malt}{\mathrm{({\it Alt'})}}
\newcommand{\mmalt}{\mathrm{({\it Alt''})}}
\newcommand{\galt}{\mathrm{({\it Alt'''})}}
\newcommand{\nulla}{\mathrm{{\it Null}}(A)}
\newcommand{\infeas}{{\cal D}}

\theoremstyle{definition}

\newtheorem{Rem}{Remark}[section]
\newtheorem{Assu}{Assumption}[section] 
\newtheorem{Prob}{Challenge}

\DeclareMathOperator*{\argmax}{\arg\!\max}

\DeclareMathOperator*{\sgn}{sgn}

\DeclareMathOperator*{\diag}{diag}
\DeclareMathOperator*{\vol}{vol}

\algnewcommand\algorithmicinput{\textbf{Input:}}
\algnewcommand\INPUT{\item[\algorithmicinput]}
\algnewcommand\algorithmicoutput{\textbf{Output:}}
\algnewcommand\OUTPUT{\item[\algorithmicoutput]}

\begin{document} 

\title{An Oblivious Ellipsoid Algorithm for Solving a System of (In)Feasible Linear Inequalities\thanks{The authors thank the associate editor and two anonymous referees for their diligent efforts and for their constructive and very helpful comments.}}   
\author{
Jourdain Lamperski\thanks{MIT Operations Research Center, 77 Massachusetts Avenue, Cambridge, MA   02139 ({mailto:  jourdain@mit.edu}).  This author's research is supported by AFOSR Grant No. FA9550-19-1-0240.} \and Robert M. Freund\thanks{MIT Sloan School of Management, 77 Massachusetts Avenue, Cambridge, MA   02139 ({mailto:  rfreund@mit.edu}).  This author's research is supported by AFOSR Grant No. FA9550-19-1-0240.} \and Michael J. Todd\thanks{Cornell University, mjt7@cornell.edu} }
\date{\today}
\maketitle

\begin{abstract} 
The ellipsoid algorithm is a fundamental algorithm for computing a solution to the system of $m$ linear inequalities 
in $n$ variables $\poi: A^{\top}x \le u$ when its set of solutions has positive volume. However, when $\poi$ is infeasible, the ellipsoid algorithm has no mechanism for proving that $(P)$ is infeasible. This is in contrast to the other two fundamental algorithms for tackling $\poi$, namely the simplex method and interior-point methods, each of which can be easily implemented in a way that either produces a solution of $\poi$ or proves that $\poi$ is infeasible by producing a solution to the alternative system $\alt: A\lambda= 0$, $u^{\top}\lambda < 0$, $\lambda \ge 0$. This paper develops an Oblivious Ellipsoid Algorithm (OEA) that either produces a solution of $\poi$ or produces a solution of $\alt$. Depending on the dimensions and on other natural condition measures, the computational complexity of the basic OEA may be worse than, the same as, or better than that of the standard ellipsoid algorithm. We also present  two modified versions of OEA, whose computational complexity is superior to that of OEA when $n \ll m$. This is achieved in the first modified version by proving infeasibility without actually producing a solution of $\alt$,
and in the second modified version by using more memory.
\end{abstract}

\section{Introduction, preliminaries, and summary of results} \label{s: intro} 
Given data $(A,u) \in \mathbb{R}^{n \times m} \times \mathbb{R}^m$, the ellipsoid algorithm is a fundamental algorithm for computing a solution to the system of linear inequalities 
\begin{equation} \nonumber 
\poi: \ \  A^{\top}x \le u
\end{equation} 
when the set of solutions $\cp := \{ x \in \mathbb{R}^n: A^{\top} x \leq u \}$ has positive volume.  However, when $\poi$ is infeasible, existing versions of the ellipsoid algorithm have no mechanism for deciding if $\poi$ is infeasible.  (We use the real number model of computation throughout this paper.  In the bit model of computation the ellipsoid method will correctly decide infeasibility even though it will not produce a solution of a dual/alternative system -- instead a volume argument is used to prove infeasibility; see \cite{gls}.)  By a \emph{certificate of infeasibility} we informally mean a mathematical object that yields a proof that $\poi$ is infeasible.  For example, when and only when $\poi$ is infeasible, there exists a solution $\lambda \in \mathbb{R}^m$ to the alternative system $\alt$ below, which we formally call a \emph{type-\1} certificate of infeasibility:\medskip

\noindent \textbf{Type-\1 Certificate of Infeasibility.}  If $\lambda \in \mathbb{R}^m$ satisfies: 
\begin{equation} \nonumber 
\alt:  \ \ \left\{ \begin{array}{rcl}A\lambda & =& 0 \\  
\lambda & \ge& 0 \\
u^{\top}\lambda & <& 0 \  ,\end{array} \right.
\end{equation} 
then it is simple to demonstrate that $\poi$ is infeasible.  We refer to a solution to $\alt$ as a \emph{type-\1} certificate of infeasibility, where \1 stands for {\em linear} because the certificate is identified with a linear inequality system, and in order to distinguish it from two other types of certificates of infeasibility for $\poi$ that will be developed herein.  We view a type-\1 certificate of infeasibility as special because -- like a solution to $\poi$ -- it is a solution to a particular linear inequality system (namely $\alt$), it does not require excessive storage ($m$ coefficients), and the computation involved in verifying $\alt$ is not excessive ($O(mn)$ operations).  

The two other fundamental algorithms for tackling $\poi$, namely the simplex algorithm and interior-point methods, each can be implemented in a way that either produces a solution to $\poi$ or certifies that $\poi$ is infeasible by producing a type-\1 certificate of infeasibility.  This has begged the question of whether such a version of the ellipsoid method can be developed \cite{mjtbordeaux}, that is, can one develop an \emph{oblivious ellipsoid algorithm} that produces a solution to $\poi$ or $\alt$, {\em without knowing} which system is feasible?  Accordingly, we consider the following two challenges, the \emph{oblivious linear certification challenge} and the \emph{oblivious determination challenge}: 

\begin{Prob}[{\bf Oblivious Linear Certification}] \label{p1} 
Develop a version of the ellipsoid algorithm that produces a feasible solution of $\poi$ when $\poi$ is feasible, and produces a type-L certificate of infeasibility, i.e., a solution of $\alt$, when $\poi$ is infeasible.  
\end{Prob} 

\begin{Prob}[{\bf Oblivious Determination}] \label{p2} 
Develop a version of the ellipsoid algorithm that produces a feasible solution of $\poi$ when $\poi$ is feasible, and proves that $\poi$ is infeasible when $\poi$ is infeasible.
\end{Prob}

When $\poi$ is feasible, both Challenges \ref{p1} and \ref{p2} require producing a solution of $\poi$.  But when $\poi$ is not feasible, Challenge \ref{p1} requires producing a type-\1 certificate of infeasibility, whereas Challenge \ref{p2} only requires proving infeasibility -- though not necessarily producing a type-\1 certificate.  It follows that any resolution of Challenge \ref{p1} is also a resolution of Challenge \ref{p2}. 

Of course, one could address Challenge \ref{p1} or Challenge \ref{p2} by running the standard ellipsoid method in parallel simultaneously on $\poi$ and $\alt$. That is, one could perform (one arithmetic operation at a time) one operation of the ellipsoid algorithm applied to $\poi$ followed by one operation of the ellipsoid algorithm applied to $\alt$, and then stop when one of the two algorithms produces a solution. (Equivalently, one could run each algorithm on a separate machine.)  However, there is an aesthetic interest in developing a single oblivious ellipsoid algorithm (which we call OEA) that will either produce a solution of $\poi$ or prove that $\poi$ is infeasible by producing a solution of $\alt$.  Such a version would elevate the ellipsoid algorithm to be ``on par'' with the other two fundamental algorithms for solving $\poi$ in this regard, namely the simplex method and interior-point methods.

Before presenting a schematic of OEA and stating our main results, we first need to develop some relevant concepts and related notation. We will make the following assumption about the data throughout this paper: 

\begin{Assu} \label{assu1}  
The conic hull of the columns of $A$ is equal to $\mathbb{R}^n$, namely $\{ A \lambda : \lambda \ge 0 \} = \mathbb{R}^n$, and each of the columns $a_1,...,a_m$ of $A$ has unit Euclidean norm. 
\end{Assu}

The first part of Assumption \ref{assu1} ensures that $\poi$ is bounded if $\poi$ is feasible. Note that Assumption \ref{assu1} implies that $m > n$ and that $A$ has rank $n$.  The second part of Assumption \ref{assu1} is without loss of generality because feasible solutions of $\poi$ do not change under positive rescaling of the constraints of $\poi$, and any zero $a_j$'s either yield redundant constraints or immediate proofs of infeasibility.

We will suppose that for each $i \in \{1,...,m\}$ we know a \emph{lower bound} $\ell_i \in \mathbb{R}$ that satisfies $x \in \cp \Rightarrow a_i\adj x \ge \ell_i$. (When $\poi$ is infeasible, any $\ell_i \in \mathbb{R}$ satisfies this implication vacuously.)  Accordingly, if $\poi$ is feasible and we know lower bounds $\ell_1, \ldots, \ell_m$, then we know how to bound $a_i^{\top}x$ for $x \in \cp$ since $u_i$ is an upper bound for $a_i^{\top}x$ over all $x \in \cp$); see Figure \ref{lbpic}.  

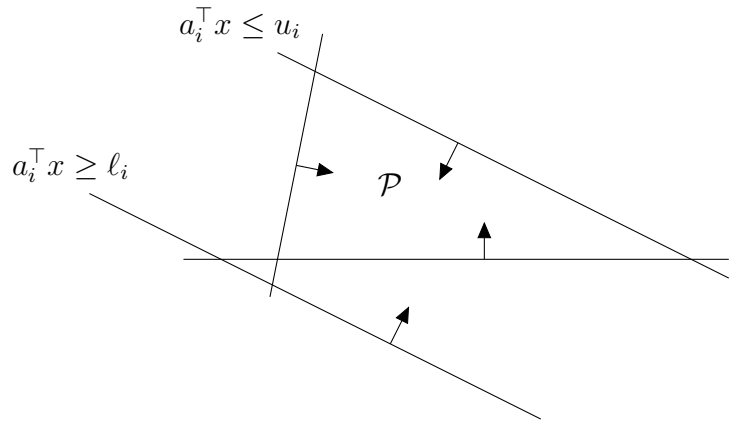
\begin{figure}[h!] 
\centering
\begin{tikzpicture}
	[scale=.5, auto=left, every node/.style={circle}] 
	
	\node (P) at (-2,2) {$\mathcal{P}$}; 
	\node (u) at (-6,6.25) {$a_i^{\top} x \leq u_i$}; 
	\node (ell) at (-10.5,2.5) {$a_i^{\top} x \geq \ell_i$}; 
	
	\draw (-7.5,0) to (7,0); 
	\draw(-5.2, -1) to (-3.8,6); 
	\draw (-5,5.5) to (7,-.5); 
	\draw (-10,1.75) to (2,-4.25); 
	
	\draw[->, -triangle 45] (0.5,0) to (0.5,1); 
	\draw[->, -triangle 45] (-4.5, 2.5) to (-3.5,2.3); 
	\draw[->, -triangle 45] (-0.2,3.1) to (-0.7,2.1); 
	\draw[->, -triangle 45] (-2, -2.25) to (-1.5,-1.25); 
	
\end{tikzpicture}
\caption{A lower bound $\ell_i$ on constraint $i$ of $\poi$.} \label{lbpic}
\end{figure} 

In fact, we will suppose more strongly that for each $i \in \{1,...,m\}$ we know $\ell_i$ and $\lambda_i \in \mathbb{R}^m$ that satisfy: 
\begin{align}
\lbi: \ \ \left \{
\begin{aligned} \nonumber 
A \lambda_i & = -a_i  \\
\lambda_i  & \geq 0 \\
-\lambda_i^{\top} u & \geq \ell_i \ , 
\end{aligned} 
\right.
\end{align} 
and observe that when $(P)$ is feasible, it follows from $\lbi$ that any $x \in \cp$ satisfies 
$$a_i^{\top} x = - \lambda_i^{\top} A^{\top} x \geq - \lambda_i^{\top} u \geq \ell_i \ ,$$
i.e., $\lambda_i$ certifies the lower bound $\ell_i$ on $a_i \adj x$ over all $x \in \cp$. We will define $\ell_i$ to be a \emph{certified lower bound for constraint $i$ of $\poi$ with certificate $\lambda_i \in \mathbb{R}^m$} if $\ell_i$ and $\lambda_i$ together satisfy $\lbi$. 

It will be convenient to collect the certified lower bounds into $\ell = (\ell_1, \ldots, \ell_m)^{\top}  \in \mathbb{R}^m$ and their certificates columnwise into a matrix $\Lambda = [\lambda_1 |  \cdots |\lambda_m] \in \mathbb{R}^{m \times m}$, and define $\ell$ to be a \emph{certified lower bound for $\poi$ with certificate matrix $\Lambda \in \mathbb{R}^{m \times m}$} if $\ell$ and $\Lambda$ satisfy 
\begin{align} 
\lb: \ \ \left \{
\begin{aligned} \nonumber 
A \Lambda  & = -A \\ 
\Lambda & \geq 0 \\ 
- \Lambda^{\top} u & \geq \ell \ ,
\end{aligned} 
\right.
\end{align} 
where the matrix inequalities $\Lambda \geq 0$ are considered entry-wise. Just as above, if $\poi$ is feasible and $\ell$ is a certified lower bound for $\poi$ with certificate $\Lambda$, then $\ell$ is a lower bound for $A^{\top} x$ over all $x \in \cp$ because for any $x \in \cp$ it holds that 
$$A^{\top} x = - \Lambda^{\top} A^{\top} x \geq - \Lambda^{\top} u \geq \ell \ .$$

In general, it is not such an easy task to construct such lower bounds $\ell$ and certificates $\Lambda$ -- short of solving systems of inequalities of size at least as large as that of $\poi$. However, in the often-occurring case when $\poi$ contains box constraints (of the form $\underline b \le x \le \bar b$), such lower bounds and certificates are quite simple to write down, which we show in Section \ref{ss: valid_lbs}.  (Recall that if $(P)$ is infeasible, then any $\ell_i$ is a lower bound. In theory, we can find a solution to $\lbi$ by obtaining a nonnegative solution to $A \hat \lambda = - a_i$ by Assumption \ref{assu1} and then adding to it a suitably large multiple of a solution to $\alt$.) 

We can use a certified lower bound $\ell$ together with an arbitrary given $d \in \mathbb{R}^m$ satisfying $d > 0$ to construct a parametrized ellipsoid $E(d,\ell)$ that contains $\cp$: 
\begin{equation}\label{reps}
\cp \subseteq E(d,\ell)  := \left \{x \in \mathbb{R}^n : (A^{\top}x -\ell)^{\top}D(A^{\top} x - u) \leq 0 \right \}  \ , \end{equation}
(where $D := \diag(d)$ is the diagonal matrix with diagonal $d$), since $x \in \cp \Rightarrow (a_i\adj x -u_i )d_i (a_i\adj x -\ell_i) \le 0$ for all $i=1, \ldots, m$.  Using some elementary algebraic manipulation, we can re-write $E(d,\ell)$ as:
\begin{equation}\label{reps2} E(d,\ell) = \left \{ x \in \mathbb{R}^n : (x-y(d,\ell))^{\top}ADA^{\top}(x - y(d,\ell)) \leq f(d,\ell) \right \} \ , 
\end{equation} 
where 
\begin{align*} 
y(d,\ell) &:= \tfrac{1}{2}(ADA^{\top})^{-1} AD(u + \ell) \ , \\
f(d,\ell) &:= \tfrac{1}{4}(u+\ell)^{\top}DA^{\top}(ADA^{\top})^{-1} AD(u+\ell)- \ell^{\top}D u \ ,
\end{align*} 
and we see from \eqref{reps2} that $y(d,\ell)$ is the center of the ellipsoid $E(d,\ell)$, $ADA^{\top}$ is the so-called shape matrix, and $\sqrt{f(d,\ell)}$ captures the scale factor of the ellipsoid.  

The representation (\ref{reps})-(\ref{reps2}) was introduced by Burrell and Todd \cite{bt85} to generate dual variables in the ellipsoid method. They developed a variant of the standard ellipsoid method with deep cuts that represented each ellipsoid in the form $E(d,\ell)$ (with $d \geq 0$, not necessarily positive); the difference was that sometimes it was necessary to update the lower bounds  before applying the standard deep cut update.

In the Oblivious Ellipsoid Algorithm that we develop in this paper, we will also update the ellipsoid $E(d,\ell)$ by updating its parameters $(d,\ell) \rightarrow (\tilde d, \tilde \ell)$ (as opposed to explicitly updating the center and shape matrix as is done in the conventional ellipsoid algorithm).  Hence $\ell$ (and its certification matrix $\Lambda$) should be thought of as parameters that are given an initial value and then are updated in the course of running the algorithm. We will also
maintain $d$ positive throughout.

Our Oblivious Ellipsoid Algorithm will update $\ell$ in synch with updates of $\Lambda$ so that the updated $\ell$ is always certified by the updated $\Lambda$.  For motivation why OEA updates $\ell$ and $\Lambda$, suppose at a given iteration we have $\ell$ that is certified by $\Lambda$ and it holds that $\ell_j$ satisfies $\ell_j > u_j$ for some $j \in \{1,...,m\}$ (so that clearly $\poi$ is infeasible). Then it is straightforward to verify (see Burrell and Todd \cite{bt85}, and also Corollary \ref{heatwave} here) that $\bar{\lambda}_j := \lambda_j + e_j$ is feasible for $\alt$ and so is a type-\1 certificate of infeasibility.  This will be our main method for constructing a type-\1 certificate of infeasibility in our algorithm, so we state this result formally as follows.

\begin{Rem}\label{p: cert_inf_const-1} 
Suppose $\ell_j$ is a certified lower bound for inequality $j$ with certificate $\lambda_j$, and that $\ell_j > u_j$.  Then $\poi$ is infeasible, and $\bar\lambda_j := \lambda_j + e_j$ is feasible for $\alt$ and hence is a type-\1 certificate of infeasibility. 
\end{Rem} 

We can also use $\ell$ and $\Lambda$ satisfying $\lb$ to construct certificates of infeasibility that are different from a type-\1 certificate.  Let us show two ways that this can be done, which we will call \emph{type-\2} and \emph{type-\3} certificates of infeasibility, respectively.   

\medskip

\noindent \textbf{Type-\2 Certificate of Infeasibility.} Let $d \in \mathbb{R}^m$ satisfying $d > 0$ be given, and let $\ell$ be a certified lower bound for $\poi$ with certificate matrix $\Lambda$. It follows from \eqref{reps} and \eqref{reps2} that if  $f(d,\ell) \leq 0$ and $A\adj y(d,\ell) \not\le u$, then $\poi$ is infeasible. Thus, $d \in \mathbb{R}^m$, $\ell \in \mathbb{R}^m$, and $\Lambda \in \mathbb{R}^{m \times m}$ that satisfy 
\begin{align} 
\begin{aligned} \nonumber
d & > 0 \\ 
A \Lambda  & = -A \\ 
\Lambda & \geq 0 \\ 
- \Lambda^{\top} u & \geq \ell \\ 
f(d,\ell) & \leq 0 \\
A\adj y(d,\ell) & \not\le u \ 
\end{aligned} 
\end{align} 
comprise a certificate of infeasibility, which we will refer to as a type-\2 certificate of infeasibility, where \2 stands for \emph{quadratic} because the fifth system above is a quadratic inequality in $\ell$. (And later in this paper, we will show how to construct a type-\1 certificate of infeasibility from a type-\2 certificate of infeasibility; see Proposition \ref{p: alg_correct_f_leq_0}.)

\medskip

\noindent \textbf{Type-\3 Certificate of Infeasibility.} Let $d \in \mathbb{R}^m$ satisfying $d > 0$ be given, and let $\ell$ be a certified lower bound for $\poi$ with certificate matrix $\Lambda$.  Suppose that $f(d,\ell) >0$, whereby from \eqref{reps2} it follows that $E(d,\ell)$ has positive volume.

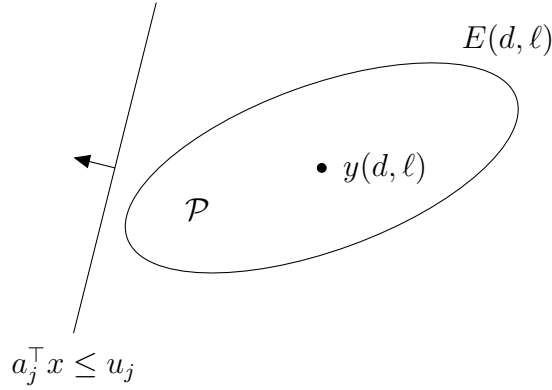
\begin{figure}[h!] 
\centering
\begin{tikzpicture}
	[scale=.55, auto=left, every node/.style={circle}]

	\node (P) at (-3,-1) {$\mathcal{P}$}; 
	\node (E) at (4.5,3.1) {$E(d,\ell)$}; 
	\node (C) at (-6,-4.8) {$a_j^{\top} x \leq u_j$}; 
	
	\draw[->, - triangle 45] (-5,0) to (-6, .25); 
	
	\node[draw, fill, scale = .3, label = right:{$y(d,\ell)$}] (o) at (0,0) {};
	\draw[rotate =20] (0,0) ellipse (5cm and 2cm);
	\draw (-6,-4) to (-4,4); 
\end{tikzpicture}
\vspace{-6mm}
\caption{All points in $E(d,\ell)$ violate constraint $j$ of $\poi$.} \label{allellviol}
\end{figure} 

\noindent It then follows from \eqref{reps} and \eqref{reps2} that if there exists $j \in \{1,...,m\}$ satisfying
$$u_j < \min_{x \in E(d,\ell)} a_j^{\top}x \ , $$
namely every point in $E(d,\ell)$ violates constraint $j$ of $\poi$, then $\poi$ is infeasible (see Figure \ref{allellviol}).  Now notice that 
$$\min_{x \in E(d,\ell)} a_j^{\top}x = a_j\adj y(d, \ell) - \sqrt{f(d,\ell)} \sqrt{a_j\adj (ADA\adj)^{-1}a_j} \ .$$  
Thus $d \in \mathbb{R}^m$, $\ell \in \mathbb{R}^m$, $\Lambda \in \mathbb{R}^{m \times m}$, and $j \in \{1,...,m\}$ that satisfy 
\begin{align} 
\begin{aligned} \nonumber 
d & > 0 \\ 
f(d,\ell) & >0 \\
A \Lambda  & = -A \\ 
\Lambda & \geq 0 \\ 
- \Lambda^{\top} u & \geq \ell \\ 
u_j & < a_j\adj y(d, \ell) - \sqrt{f(d,\ell)} \sqrt{a_j\adj (ADA\adj)^{-1}a_j} 
\end{aligned} 
\end{align} 
comprise a certificate of infeasibility of $\poi$, which we will refer to as a type-\3 certificate of infeasibility, where \3 stands for {\em ellipsoid} because the bound arises from minimization over the ellipsoid $E(d,\ell)$ as just described.  Burrell and Todd \cite{bt85} show how to construct a type-\1 certificate of infeasibility from a type-\3 certificate of infeasibility, which we will review in Proposition \ref{p: cert_const} and Corollary \ref{heatwave}. \medskip

\noindent We note that there can of course be many other types of certificates of infeasibility beyond the three types just described.

\subsection{Schematic of the Oblivious Ellipsoid Algorithm} 
Algorithm \ref{a: schem} below is an informal schematic of our Oblivious Ellipsoid Algorithm.  (For the full algorithm description of OEA, see Algorithm \ref{alg: 1} and the surrounding discussion.)

\begin{algorithm}[H] 
\caption{Schematic of Oblivious Ellipsoid Algorithm (OEA)} \label{a: schem} 
\begin{algorithmic}[1] 
\INPUT data $A$ and $u$, certified lower bound $\ell$ for $\poi$ with certificate matrix $\Lambda$, and $d > 0$. 
\medskip
\State Compute $y(d,\ell)$. If $A^{\top}y(d,\ell) \leq u$, then Return $y(d,\ell)$ as a solution of $\poi$ and Stop. \label{sch1}
\State Compute $f(d,\ell)$. If $f(d,\ell) \leq 0$, then construct and Return a certificate of infeasibility and Stop.\label{sch2}
\State Compute the most violated constraint: $j \leftarrow \argmax_{i \in \{1,...,m\}} a_i^{\top} y(d,\ell) - u_i$. \label{sch3}
\State (Possibly) update certificate $\lambda_j$ if its best lower bound can be improved. \label{sch4}
\State If $\min_{x \in E(d,\ell)} a_j^{\top} x_j > u_j$, then construct and Return a certificate of infeasibility and Stop.\label{sch5}
\State Update ellipsoid $E(d,\ell)$ by updating $(d,\ell) \rightarrow (\tilde d, \tilde \ell)$. \label{sch6}
\State Re-set $(d,\ell) \leftarrow (\tilde d, \tilde \ell)$ and Goto Step \ref{sch1}. \label{sch7}
\end{algorithmic}
\end{algorithm} 

The iterates of Algorithm \ref{a: schem} are $d$, $\ell$, and $\Lambda$. In Step \ref{sch1}, we perform a ``standard'' ellipsoid algorithm step where we check if the center $y(d,\ell)$ of the ellipsoid $E(d,\ell)$ is a feasible solution of $\poi$, and if so we output $y(d,\ell)$ and stop. If we proceed to Step \ref{sch2}, then $A\adj y(d,\ell) \not\le u$. In Step \ref{sch2}, we check if $f(d,\ell) \leq 0$, and if this holds, then $d$, $\ell$, and $\Lambda$ comprise a type-\2 certificate of infeasibility, from which we can construct a type-\1 certificate of infeasibility (as will be shown in Proposition \ref{p: alg_correct_f_leq_0}).  In Step \ref{sch3} we perform another standard ellipsoid algorithm step wherein we compute the index of the most violated constraint.  (Actually, in the standard ellipsoid method it is sufficient to compute the index of any violated constraint, but computing the most violated constraint will be crucial for establishing the convergence guarantee of the Oblivious Ellipsoid Algorithm when $\poi$ is infeasible.)  In Step \ref{sch4}, we possibly update the lower bound certificate $\lambda_j$ if the update certifies a better lower bound than the largest lower bound currently certified.  In Step \ref{sch5}, we check if $\min_{x \in E(d,\ell)} a_j^{\top} x_j > u_j$, and if this condition is satisfied, then $d$, $\ell$, and $\Lambda$ comprise a type-\3 certificate of infeasibility, from which we can construct and return a type-\1 certificate of infeasibility (as will be shown in Proposition \ref{p: cert_const} and Corollary \ref{heatwave}).   In Step \ref{sch6}, we update the ellipsoid $E(d,\ell)$ by computing new values $(\tilde d, \tilde \ell)$ of the parameters of $E(\cdot, \cdot)$ which replace the current values $(d,\ell)$ in Step \ref{sch7}.  

\subsection{Summary of Main Results} 
We briefly summarize our main results concerning the Oblivious Ellipsoid Algorithm. In the case when $\poi$ is infeasible, OEA will compute a type-\1 certificate of infeasibility in $$\left \lfloor 2m(m+1) \ln \left( \frac{m+1}{2m}\frac{\lVert u - \ell\rVert }{\tau(A,u)} \right) \right \rfloor$$ 
iterations, where $\ell$ is the initial lower bound for $\poi$ certified by the initial $\Lambda$, and $\tee$ is a geometric condition number that naturally captures the extent of feasibility or infeasibility of $\poi$; see Theorem \ref{t: infeas_guar} as well as Corollary \ref{c: infeas_box_guar} which specializes the above bound to the case where $\poi$ contains box constraints.  Each iteration of OEA requires $O(m^2)$ arithmetic
operations in a straightforward implementation, whence the total computational complexity of OEA when $\poi$ is infeasible is $O(m^4 \ln \frac{1}{\tau})$. (We assume for simplicity here that $u$ and $u - \ell$ are $\Theta(1)$.)

In the case when $\poi$ is feasible, OEA will compute a solution of $\poi$ in 
$$\left \lfloor 2n(m+1) \ln \left( \frac{\lVert u - \ell\rVert }{2\rho(A) \tau(A,u)} \right) \right \rfloor$$ 
iterations, where $\ell$ is the initial certified lower bound for $\poi$ certified by the initial $\Lambda$, $\rho(A)$ is a geometric condition measure 
which captures the distance to unboundedness of $\poi$, and $\tau(A,u)$ is the geometric condition measure mentioned above (which corresponds to the radius of the largest inscribed ball in the feasible region $\cp$ in the feasible case); see Theorem \ref{t: feas_guar} as well as Corollary \ref{c: feas_box_guar} which specializes the above bound to the case where $\poi$ contains box constraints (in which case $\rho(A)$ plays no role).  Since each iteration of OEA requires $O(m^2)$ operations, the total computational complexity of OEA when $\poi$ is feasible is $O(m^3n \ln \frac{1}{\rho \tau})$.

The iteration bound in the feasible case follows from standard volume-reduction arguments.  However, in the infeasible case the iteration bound follows from a proof that at each iteration a novel potential function is reduced.  The introduction of this potential function is another contribution of our paper; see Section \ref{s: infeasible}.  We emphasize that the algorithm we develop does not know whether the problem is feasible or infeasible (hence oblivious); however, in either case, it makes progress at each iteration towards determining that fact (both in volume reduction and in the potential function). We leave open the possibility of a less oblivious algorithm that strives to make greater progress in decreasing the volume or the potential function.

Let us compare the computational complexity bounds above to the strategy of running the standard ellipsoid algorithm in parallel simultaneously on $\poi$ and $\alt$ (which we denote by the acronym ``SEAP'' here and in Table \ref{comp}). Our comparison assumes a straightforward implementation of the algorithms involved; however, we show how to reduce the number of operations per iteration of each algorithm using a more complicated implementation in Appendix \ref{impl_Alt}.  
The first two rows of Table \ref{comp} present the relevant bounds in the comparison. The standard ellipsoid scheme SEAP has potentially superior computational complexity over OEA -- $ O(mn^3\ln \frac{1}{\rho \tau})$ rather than $O(m^3n \ln \frac{1}{\rho \tau})$ total operations when $\poi$ is feasible, and $O(m(m-n)^3 \ln \frac{1}{\tau})$ rather than $O(m^4 \ln \frac{1}{\tau})$ when $\poi$ is infeasible and $\frac{m}{\rho}$ is sufficiently small. For instances when $\poi$ is infeasible, $m = O(m - n)$, and $\frac{m}{\rho}$ is sufficiently small, the two algorithms have the same computational complexity, namely $O(m^4 \ln \frac{1}{\tau})$. And for instances when $\poi$ is infeasible, $m = O(m - n)$, and $\frac{m}{\rho}$ is sufficiently large, OEA has potentially superior computational complexity over SEAP -- $O(m^4 \ln \frac{1}{\tau})$ rather than $O(m^4 \ln \frac{m}{\rho \tau})$. 

{\tiny
\begin{table}[h!]
\begin{center}
\scalebox{0.79}{
\begin{tabular}{| l || c | c | c | c | c | c |}
\hline
\multirow{4}{*}{Algorithm} & \multicolumn{2}{ c |}{} & \multicolumn{2}{ c |}{Operations} & \multicolumn{2}{ c |}{}  \\
& \multicolumn{2}{ c |}{Number of} & \multicolumn{2}{ c |}{per Iteration} & \multicolumn{2}{ c |}{Total Number} \\
& \multicolumn{2}{ c |}{Iterations} & \multicolumn{2}{ c |}{(standard implementation)} & \multicolumn{2}{ c |}{of Operations} \\
\cline{2-7} 
& $\poi$ is & $\poi$ is & $\poi$ is & $\poi$ is & $\poi$ is & $\poi$ is \\ 
& Feasible & Infeasible & Feasible & Infeasible &Feasible & Infeasible \\ 
\hline 
\hline
SEAP & $O(n^2 \ln \frac{1}{\rho \tau})$ & $O((m-n)^2 \ln \frac{m}{\rho \tau})$ & $O(mn)$ & $O(m(m-n))$ & $O(mn^3 \ln \frac{1}{\rho \tau})$ & $O(m(m-n)^3 \ln \frac{m}{\rho \tau})$  \\  
\hline 
OEA & $O(mn \ln \frac{1}{\rho \tau})$ & $O(m^2  \ln \frac{1}{ \tau})$ & $O(m^2)$ & $O(m^2)$ & $O(m^3n \ln \frac{1}{\rho\tau})$ & $O(m^4 \ln \frac{1}{\tau})$ \\
\hline 
OEA-No-Alt & $O(mn  \ln \frac{1}{\rho \tau} )$ & $O(m^2   \ln \frac{1}{ \tau} )$ & $O(mn)$ & $O(mn)$ & $O(m^2n^2 \ln \frac{1}{\rho \tau})$ & $O(m^3n \ln \frac{1}{\tau})$  \\
\hline 
OEA-MM & $O(mn  \ln \frac{1}{\rho \tau} )$ & $O(m^2   \ln \frac{1}{ \tau} )$ & $O(mn)$ & $O(mn)$ & $O(m^2n^2 \ln \frac{1}{\rho \tau})$ & $O(m^3n \ln \frac{1}{\tau})$  \\
\hline 
\end{tabular}
}
\end{center}
\caption{Computational complexity comparison of Standard Ellipsoid Algorithm in Parallel (SEAP), Oblivious Ellipsoid Algorithm (OEA), and the modified versions OEA-No-Alt and OEA-MM.} \label{comp} 
\end{table} 
}

We also present two modified versions of OEA, which we call OEA-No-Alt and OEA-MM, that require a smaller number of operations per iteration than OEA -- $O(mn)$ instead of $O(m^2)$.  Here we briefly describe these versions and their computational complexity.  Recall from earlier in this section that OEA maintains at each iteration a lower bound vector $\ell$ that is certified by a corresponding certificate matrix $\Lambda$, and that $\ell$ and $\Lambda$ are updated at each iteration. The increased work per iteration of OEA -- $O(m^2)$ instead of $O(mn)$ -- is the result of maintaining/updating the matrix $\Lambda$. It turns out that the certificate matrix $\Lambda$ is not actually used anywhere in the computations in the algorithm; rather its sole purpose is to produce a Type-\1 certificate of infeasibility (a solution of $\alt$) after such infeasibility is detected.  If one is only interested in solving Challenge \ref{p2}, i.e., correctly detecting infeasibility (but not necessarily producing a solution of $\alt$), then the updates of $\Lambda$ in OEA can be removed from the steps of the algorithm, which simplifies the work per iteration and yields $O(mn)$ operations per iteration as opposed to $O(m^2)$ operations per iteration for OEA.  The resulting modified method is called OEA-No-Alt because it does not produce a solution of $\alt$, and its computational complexity bounds are shown in the third row of Table \ref{comp}.  We point out that in the case when $\poi$ is infeasible and $ n \ll m$, the last column of the table indicates that the computational complexity of OEA-No-Alt for proving infeasibility (by correctly detecting that $\poi$ is infeasible) is $O(m^3n \ln \frac{1}{\tau})$, which is superior to the $O(m^4 \ln \frac{m}{\rho \tau})$ computational complexity of SEAP -- at the expense of not producing a solution to $\alt$.  We refer the reader to Section \ref{oeas} for details.  (However, we present a specific reformulation of $\poi$ in Appendix \ref{impl_Alt} that reduces the number of operations required in each iteration of both SEAP and OEA-No-Alt, eliminating this complexity advantage at the cost of possibly increasing the values of the condition measure $\tau$ and $\rho$ of the problem.)

The second modified version of OEA that we develop is called OEA-MM because it uses more memory.  This version of OEA postpones updating the certificate matrix $\Lambda$ until after OEA detects infeasibility in order to again reduce the operation complexity at each iteration from $O(m^2)$ to $O(mn)$. However, OEA-MM requires more memory storage for this post-processing step. Again we refer the reader to Section \ref{oeas} for details.

\medskip

\noindent \textbf{Differences between the Oblivious Ellipsoid Algorithm and the Standard Ellipsoid Algorithm.} We did not see a way to use a standard version of the ellipsoid algorithm to solve Challenge \ref{p1} or \ref{p2}.  In particular, standard versions are designed to decrease the volume of the ellipsoid as much as possible at each iteration (by computing the minimum volume ellipsoid that contains the current half-ellipsoid), and we found this to be detrimental to establishing any type of guarantee when $\poi$ is infeasible. Accordingly, we develop an alternative way to update ellipsoids (see Remark \ref{r: diff1}) that sufficiently decreases the volume (to obtain a guarantee when $\poi$ is feasible) while also decreasing the value of a certain potential function that we introduce that is related to infeasibility measures.  Like the volume of a full-dimensional polytope, the potential is bounded from below, which allows us to establish a guarantee when $\poi$ is infeasible; see Section \ref{s: infeasible} for the details. 


\subsection{Literature Review}
The ellipsoid method was introduced by Yudin and Nemirovsky \cite{yn} in their study of the complexity of convex optimization, 
and independently by Shor \cite{shor}, and then famously
used by Khachiyan \cite{kha} to show that linear programming (in the bit model) is polynomial-time bounded. Both Yudin-Nemirovsky and
Khachiyan used a varying coordinate system to describe their ellipsoids, but G\'{a}cs and Lov\'{a}sz in their exposition of the method \cite{gl}
and almost all subsequent authors used the representation $\{x \in \mathbb{R}^n: (x - y)^{\top} G^{-1} (x - y) \leq 1 \}$ in terms of the center $y$ and the
shape matrix inverse $G$. Many authors developed improvements involving deep and two-sided cuts; see the survey paper \cite{bgt} and its references.
Most research concentrated on linear programming, although there was a substantial research effort devoted to consequences in
combinatorial optimization (see Gr\"{o}tschel, Lov\'{a}sz, and Schrijver \cite{gls}), and Ecker and Kupferschmid showed the effectiveness
of the method on medium-sized nonlinear programming problems \cite{ek}.

Here we are concerned with the linear case, and indeed just with linear inequalities. In the literature on linear programming, most variants of the ellipsoid method just describe
the updates to the center $y$ and the shape matrix $G$. In proving that the formulae for two-sided cut variants gave minimum-volume ellipsoids, Todd
\cite{tomv} showed that the new quadratic inequality was a convex combination of that defining the old ellipsoid and one requiring the solution to
lie between the two hyperplanes defining the two-sided cut. This insight led later to the Burrell-Todd representation described above \cite{bt85}.

We also mention that there are variants of the ellipsoid method which enclose the feasible region in a sequence of convex bodies whose volumes decrease geometrically.  One such method is the ``simplex'' method of Yamnitsky and Levin \cite{yl}, which uses simplices instead of ellipsoids as the fundamental class of convex bodies in the algorithm. The guaranteed volume reduction of their method is smaller than that of the standard ellipsoid method ($O(\exp(n^{-2}))$ instead of $O(\exp(n^{-1}))$), but may be better than that of OEA when $m \gg n^2$. Similar to the Burrell-Todd variant of the ellipsoid method and also OEA, their method iteratively maintains a certificate that the current simplex contains the feasible region, but we are not aware of research on its complexity in detecting infeasibility.

\subsection{Organization}  
In Section \ref{ss: valid_lbs} we show how to easily construct initial lower bounds and certificates when $\poi$ contains box constraints.  In Section \ref{s: geometry} we further review the ellipsoid parametrization (namely \eqref{reps} and \eqref{reps2}) of \cite{bt85} and we introduce ellipsoid slab radii, which will be an important geometric concept in the setting when $\poi$ is infeasible.  In Section \ref{s: cond_num} we introduce the condition number $\tee$ that captures the extent of feasibility or infeasibility of the system $\poi$, and we also introduce the condition measure $\cc$ that measures the distance to unboundedness of $\poi$. In Section \ref{s: updating_certificates} we review and further develop a method for updating certificates for lower bounds developed initially in \cite{bt85}.  In Section \ref{s: update} we develop our mechanism for updating the ellipsoids from one iteration to the next.   In Section \ref{s: alg} we formally state our algorithm along with convergence guarantees for the feasible and the infeasible cases.  Lastly, in Section \ref{oeas} we present two modified versions of OEA, which we denote as OEA-No-Alt and OEA-MM, together with their complexity analysis.  Most of the proofs are in the appendices at the end of the paper. 

\subsection{Notation}
The $\ell_p$ norm is denoted $\|\cdot\|_p$ for $1 \le p \le \infty$, and the operator norm of a matrix $M$ is denoted by $\|M\|_{a,b} = \max_{\|v\|_a =1} \|Mv\|_b$.  For convenience we denote the Euclidean ($\ell_2$) norm simply by $\|\cdot\|$.  For $d\in \mathbb{R}^m$, we use $D$ to denote the diagonal matrix whose diagonal entries correspond to the entries of $d$.  If not obvious from context, we use $0_k$ to denote the $k$-dimensional vector of zeros, and $I_{k \times k}$ to denote the identity matrix in $\mathbb{R}^{k \times k}$.  Let $e_i$ denote the $i^{\mathrm{th}}$ unit vector, whose dimension is dictated by context, and let $e = (1, \ldots, 1)^\top$, whose dimension is also dictated by context.  We use $[k] := \{1, \ldots, k\}$.  For a given $k$-dimensional vector $v$, the positive and negative componentwise parts of $v$ are denoted by $v^+$ and $v^-$, respectively, and satisfy $v^+ \ge 0$, $v^- \ge 0$, $v = v^+ - v^-$, and $(v^+)^{\top} v^- = 0$. To save physical space, we use the notation $[ u ; v ; w]$ to denote the concatenation of column vectors $u,v,w$ into a single new column vector.

\section{Initializing lower bounds and certificates for systems with box constraints} \label{ss: valid_lbs} 
The variant of the ellipsoid method that we develop in this paper is premised on having an initial vector of lower bounds $\ell$ with associated initial certificate matrix $\Lambda$ for the linear inequalities defining $\poi$.  In general, it is not clear how to construct such lower bounds and certificates -- short of solving related systems of inequalities of size at least as large as that of the original system.  However, in the often-occurring case when the linear inequality system defining $\poi$ contains box constraints, such lower bounds and certificates are easy to write down.  Suppose that the linear inequality system is given with box constraints, namely:
\begin{align} 
(P_B): \ \ \left \{ 
\begin{aligned} \nonumber 
 \hat{A}^{\top} x & \leq \hat{u} \\ 
 x & \leq \overline b \\ 
 x & \geq \underline b \ , \\
\end{aligned} 
\right.
\end{align} 
for given data $(\hat A, \hat u, \underline b, \overline b) \in \mathbb{R}^{n \times \hat m} \times \mathbb{R}^{ \hat m} \times \mathbb{R}^{n} \times \mathbb{R}^{n}$ satisfying $\underline b \leq \overline b$.  We can re-write system $(P_B)$ in the format $A^{\top}x \le u$ by defining 
\begin{align} 
& A := 
\begin{bmatrix}
 \hat{A} &  I_{n \times n} & -I_{n \times n}
\end{bmatrix} \label{e: A_box} \\ 
& u := 
\begin{bmatrix} 
\hat{u} \ ; \ \overline b \ ; \  -\underline b
\end{bmatrix} \ ,  \label{e: u_box}
\end{align}
and we can assume without loss of generality that:
\begin{equation}\label{drew} \hat u_i \ \ \leq \ \ \max_{\underline b \leq x \leq \overline b} {a}_i \adj x \ \ = \ \  (-(a_i)^-)^{\top} \underline b + ((a_i)^+)^{\top} \overline b \\ \ \ \  \mbox{for} \ i \in [\hat m]  \end{equation}
(as otherwise constraint $i$ would be redundant and can be removed).  Let us now see how to conveniently write down certified lower bounds and certificates for the system $A^{\top}x \le u$ defined above.  For $i \in [\hat{m}]$ define $\hat{\ell}_i \in \mathbb{R}$ to be:
\begin{equation}\label{cliff}\hat{\ell}_i := \  \min_{\underline b \leq x \leq \overline b} \hat{a}_i^{\top} x \ = \ (-(a_i)^-)^{\top} \overline b + ((a_i)^+)^{\top} \underline b \ ,  \end{equation}
and $\hat \ell := (\hat\ell_1, \ldots, \hat\ell_m)$; then it is straightforward to show that 
\begin{equation} \label{e: l_box} 
\ell := 
\begin{bmatrix} 
\hat{\ell} \ ; \ \underline b \ ; \ -\overline b 
\end{bmatrix}
\end{equation}  
is a valid lower bound vector for the system $A^{\top}x \le u$ defined above in (\ref{e: A_box})-(\ref{e: u_box}).  We construct the certificate $\lambda_i$ of constraint $i$ of $A\adj x \le u$ as follows. For $i = 1,...,\hat{m}$, define
\begin{equation} \label{e: lam1_box} 
\lambda_i = 
\begin{bmatrix} 
0_{\hat{m}} \ ; \ \hat{a}_i^- \ ; \ \hat{a}_i^+
\end{bmatrix} \ , 
\end{equation} 
for $i = \hat{m}+1,..., \hat{m}+n$, define 
\begin{equation} \label{e: lam2_box} 
\lambda_i = 
\begin{bmatrix} 
0_{\hat{m}} \ ; \ 0_n \ ; \ e_i 
\end{bmatrix} \ ,
\end{equation} 
and for $i = \hat{m}+n+1,...,\hat{m}+2n$, define
\begin{equation} \label{e: lam3_box} 
\lambda_i = 
\begin{bmatrix} 
0_{\hat{m}} \ ; \ e_i \ ; \ 0_n
\end{bmatrix} \ .
\end{equation} 
It is then straightforward to check that $\Lambda = [\lambda_1 | \ldots | \lambda_m]$ defined by (\ref{e: lam1_box}-\ref{e: lam3_box}) is a certificate for the lower bounds $\ell$ defined in (\ref{e: l_box}).  In summary, given a system of inequalities with box constraints $(P_B)$, we can conveniently re-write the system in the format $A^{\top}x \le u$ and we can easily construct initial certified lower bounds $\ell$ along with an associated certificate matrix $\Lambda$.  We can also assume without loss of generality that $l_i \leq u_i$ for $i \in [m]$, for otherwise we can easily construct a type-L certificate
of infeasibility.

Last of all, and somewhat separately, we will need the following result which is straightforward to show as a consequence of \eqref{drew}, \eqref{cliff}, and Assumption \ref{assu1}: 
\begin{equation} \label{e: box_bound} 
\lVert u - \ell \rVert \leq \left( \sqrt{\hat{m} + 2 }  \right) \|  \overline b - \underline b \| \ ,
\end{equation} 
where $u$ and $\ell$ are defined as in (\ref{e: u_box}) and (\ref{e: l_box}), respectively.

\section{Containing ellipsoids and ellipsoid slab radii} \label{s: geometry} 
We consider the ellipsoid parameterization originally developed in Burrell and Todd \cite{bt85} and introduce the geometric notion of ellipsoid slab radii. For clarity and convenience, we re-present some definitions from Section \ref{s: intro}.

For $\ell \in \mathbb{R}^m$ and $d \in \mathbb{R}^m$ with $d > 0$, recall from \eqref{reps} the ellipsoid $E(d,\ell)$:
\begin{equation}\label{sunny}E(d,\ell) := \left \{x \in \mathbb{R}^n : \left(A^{\top}x - \ell \right)^{\top} D \left(A^{\top} x - u \right) \leq 0 \right \}  \end{equation}
(where $D$ is the diagonal matrix corresponding to $d$), and note that $E(d,\ell)$ has the following properties:  
\begin{enumerate} 
\item $E(d,\ell)$ is bounded; indeed from Assumption \ref{assu1}, the columns of $A$ span $\mathbb{R}^n$ and hence $ADA^{\top}$ is positive definite. 
\item If $\ell$ is a certified lower bound for $\poi$ with some certificate $\Lambda$, then $\poi$ is contained in the ellipsoid $E(d,\ell)$. 
\item The ellipsoid $E(d,\ell)$ is invariant under positive scaling of $d$, that is, $E(d,\ell) = E(\alpha d, \ell)$ for any $\alpha >0$. 
\end{enumerate} 

Let us define the following quantities:
\begin{align*} 
 r(\ell) &:= \tfrac{1}{2}(u+\ell) \ ,  \\ 
 v(\ell) &:= \tfrac{1}{2}(u-\ell) \ ,  \\ 
 B(d) &:= ADA^{\top} \ ,  \\ 
 y(d,\ell) &:= B(d)^{-1} ADr(\ell) \ ,  \\ 
 t(d,\ell) &:= A^{\top} y(d,\ell) - r(\ell) \ ,  \\
 f(d,\ell) &:= v(\ell)^{\top}Dv(\ell) - t(d,\ell)^{\top} D t(d,\ell) \ ,
\end{align*} 
and notice that $y(d,\ell)$ and $t(d,\ell)$ are invariant under positive scaling of $d$. Here $r(\ell)$ and $v(\ell)$ are the center and ``radius'' of the line segment between $u$ and $\ell$, respectively.  Note that the parameterization in the above objects is over $d$ and $\ell$, as we consider the data $A$ and $u$ of the linear inequality system to be fixed.  It is straightforward to verify the following alternative characterization of $E(d,\ell)$ using the above quantities:
\begin{equation}\label{tgif}E(d,\ell) = \left \{ x \in \mathbb{R}^n : \left(x - y(d,\ell) \right)^{\top} B(d) \left(x - y(d,\ell) \right) \leq f(d,\ell) \right \} \ . \end{equation}
Here we see that $y(d,\ell)$ is the center of $E(d,\ell)$.  Furthermore, $E(d,\ell)$ has positive volume when $f(d,\ell) > 0$, is the point set $\{y(d,\ell)\}$ when $f(d,\ell) = 0$, and is the empty set when $f(d,\ell) < 0$. 

\begin{Rem} \label{r: scaling_remark} 
When $E(d,\ell)$ has positive volume, i.e., $f(d,\ell) >0$, and in light of the fact that $E(d,\ell)$ is invariant under positive scalings of $d$, we will often (for arithmetic convenience) rescale $d$ to $d \leftarrow \frac{1}{f(d,\ell)} d,$ in order for $d$ to satisfy $f(d,\ell) = 1$.  
\end{Rem} 
  
Now suppose that $E(d,\ell)$ has positive volume, i.e., $f(d,\ell) > 0$.  For each index $i \in [m]$ define the \emph{ellipsoid slab radius} $\gamma_i(d, \ell)$ as:
\begin{alignat*}{2}
\gamma_i(d, \ell)  := & \min_\gamma && \hspace{2mm} \gamma \\
&\text{ s.t.} && \hspace{2mm} E(d,\ell) \subseteq \{ x \in \mathbb{R}^n: \lvert a_i^{\top} x - a_i^{\top} y(d,\ell) \rvert \leq \gamma \} \ ,
\end{alignat*} 
and is so named because $\gamma_i(d,\ell)$ is the radius $\gamma$ of the smallest \emph{slab} of the form $\{ x \in \mathbb{R}^n : \lvert a_i^{\top} x - a_i^{\top} y(d,\ell) \rvert \leq \gamma \}$ containing the ellipsoid $E(d,\ell)$.  The ellipsoid slab radius  $\gamma_i(d,\ell)$ has the following properties: 
\begin{enumerate} 
\item $\gamma_i(d,\ell)$ is invariant under positive scaling of $d$; this is because $E(d,\ell)$ and $y(d,\ell)$ are invariant under positive scaling of $d$, and 
\item $\gamma_i(d,\ell)$ is alternatively characterized as:
\begin{equation} \label{running}
\gamma_i(d, \ell) \  = \ \max_x   \left \{ a_i^{\top} (x-y(d,\ell)) : x \in E(d,\ell) \right \}  \ = \  \sqrt{f(d,\ell) a_i^{\top} (ADA^{\top})^{-1} a_i } \ .
\end{equation} 

\end{enumerate} 

Notice that we have only defined $\gamma_i(d,\ell)$ when $E(d,\ell)$ has positive volume, namely $f(d,\ell) >0$.  Of course, we could extend the notions above to the case when $f(d,\ell) = 0$ (whereby $E(d,\ell)$ is the point set $\{y(d,\ell)\}$ and $\gamma_i(d,\ell) = 0$), but this will not be needed.

\section{Condition Measures of Feasibility, Infeasibility, and Boundedness} \label{s: cond_num} 

We introduce the condition measure $\tee$ which will be used to measure the extent of feasibility or infeasibility of $\poi$.  Define:
\begin{equation}\label{tau} \tau(A,u) := |z^*| \  \ \ \ \mathrm{where~~~~} z^* := \max_{x \in \mathbb{R}^n} \min_{i \in [m]}  \ (  u_i - a_i^{\top} x) \ . \end{equation}  
The following proposition lists some of the relevant properties of $\tee$.  Here we use the parametric notation $\mathcal{P}_v := \{ x \in \mathbb{R}^n : A^{\top}x \le v\}$ to keep our statements simple.

\begin{Prop}\label{taulist} Basic properties of $\tee$
\begin{enumerate}
\item[(a)] If $\mathcal{P} \ne \emptyset$, then $\tee = z^* \ge 0$ and $\tee$ is the radius of the largest $\ell_2$ ball contained in $\mathcal{P}$.
\item[(b)] If $\mathcal{P} = \emptyset$, then $\tee = -z^* > 0$ is the smallest scalar $\theta$ for which the right-hand-side perturbed system $A^{\top}x \le u + \theta e$ has a feasible solution.
\item[(c)] If $\tee = 0$, then the linear inequality system $\poi$ is ``ill-posed'' in the following sense:  for any $\varepsilon >0$ there exist perturbations $\Delta b_{\mathrm{feas}}$ and $\Delta b_{\mathrm{infeas}}$ with $\|\Delta b_{\mathrm{feas}}\|_\infty \le \varepsilon$, $\|\Delta b_{\mathrm{infeas}}\|_\infty \le \varepsilon$ for which $\mathcal{P}_{u+\Delta b_{\mathrm{feas}}}$ has a strict (interior) solution and  $\mathcal{P}_{u+\Delta b_{\mathrm{infeas}}} = \emptyset$. 
\end{enumerate} 
\end{Prop} \qed

\noindent Item (a) above uses the second part of Assumption \ref{assu1} that the columns of $A$ have unit $\ell_2$ norm.  Regarding (c), the concept of ``ill-posedness'' in this context was first developed by Renegar \cite{Reneg94, Reneg95a, Reneg95b}; see Appendix \ref{jimbo} for further discussion of connections between $\tee$ and Renegar's condition measure $\bar\rho(d)$ and a proof that $\tee \ge \bar\rho(d)$.  Also, $\tee$ is constructed in a similar spirit to the condition number ${\cal C}(A)$ for homogeneous linear inequalities developed by Cheung and Cucker \cite{cc2001}. 

Notice in the case when $\poi$ is infeasible that for any $x \in \mathbb{R}^n$ there is some constraint that is violated by at least $\tee$, namely there exists $i \in [m]$ for which $a_i^{\top}x \ge u_i + \tee$.

We also introduce the following condition measure (in the spirit of Renegar \cite{Reneg94}) which captures the extent to which $\mathcal{P}$ is close to unbounded: 
\begin{equation}\label{fridays}
\cc := \min_{\Delta A \in \mathbb{R}^{n \times m}}\{ \| \Delta A\|_{1,2} : \mathrm{there~exists~} v \ne 0 \ \mathrm{satisfying} \ [A + \Delta 
A]\adj v \le 0 \} \ . 
\end{equation}  
Indeed, observe that $\rho(A)$ is the $ \| \cdot \|_{1,2}$ operator norm of the smallest perturbation of the matrix $A$ for which the perturbed feasible region becomes unbounded and hence violates Assumption \ref{assu1}.

\section{Updating certificates for lower bounds, and constructing certificates of infeasibility} \label{s: updating_certificates} 
Let $d > 0$, and let $\ell$ be a certified lower bound for $\poi$ with certificate matrix $\Lambda$, and suppose that $E(d,\ell)$ has positive volume, i.e., $f(d,\ell) > 0$.  Let $i \in [m]$ be given.  From the definition of the slab radius $\gamma_i(d,\ell)$, the ellipsoid $E(d,\ell)$ is contained in the half-space $\{x \in \mathbb{R}^n : a_i^{\top} x \geq a_i^{\top} y(d,\ell) - \gamma_i(d,\ell) \}$, and therefore when $\poi$ is feasible: 
$$x \in \cp \ \Rightarrow \ x \in E(d, \ell) \ \Rightarrow \ a_i^{\top} x \geq a_i^{\top} y(d,\ell) - \gamma_i(d,\ell) \ , $$ 
and it follows that 
\[
L_i := a_i^{\top} y(d,\ell)- \gamma_i(d,\ell)
\]
is a valid lower bound on constraint $i$ of $\poi$.  This leads to the question of whether and how can one construct a certificate $\tilde\lambda_i$ for the lower bound $L_i$?  This question was answered in the affirmative in Burrell and Todd \cite{bt85}, and we 
present their solution in the following proposition which shows how to use $d$, $\ell$, and $\Lambda$ to construct a certificate $\tilde{\lambda}_i \in \mathbb{R}^m$ for the lower bound $L_i$. (Note that it is possible that $L_i \leq \ell_i$.)

\begin{Prop} \label{p: cert_const} {\bf (see \cite{bt85})}
Let $d > 0$, and $\ell$ be a certified lower bound for $\poi$ with certificate matrix $\Lambda$, suppose $E(d,\ell)$ has positive volume, and suppose that $d$ has been rescaled so that $f(d,\ell) = 1$.  Let $i \in [m]$ be given, and define $L_i := a_i^{\top} y(d,\ell)- \gamma_i(d,\ell)$, and 
\begin{equation}\label{today}
\begin{array}{rl}
& \hat{\lambda}_i := \gamma_i(d,\ell) D t(d,\ell) - DA^{\top} B(d)^{-1} a_i \ , \\ 
& \tilde{\lambda}_i := \Lambda \hat{\lambda}_i^- + \hat{\lambda}_i^+ \ .
\end{array}
\end{equation} 
Then $L_i$ is also a certified lower bound on constraint $i$ of $\poi$, with certificate $\tilde{\lambda}_i$.  In particular, it holds that $-\tilde \lambda_i^{\top} u \ge L_i$. \qed
\end{Prop} 

For completeness as well for consistency with the notation used in this paper, we present a proof of Proposition \ref{p: cert_const} in Appendix \ref{app-119}.  

As a corollary, we can construct a type-\1 certificate of infeasibility from a type-\3 certificate of infeasibility: 

\begin{Cor} \label{heatwave} 
Under the set-up of Proposition \ref{p: cert_const}, if $ L_i :=  a_i^{\top} y(d,\ell)- \gamma_i(d,\ell) > u_i$, then $\bar\lambda_i := \tilde\lambda_i + e_i$ is a type-\1 certificate of infeasibility. 
\end{Cor}

\begin{proof}  From Proposition \ref{p: cert_const} it holds that $\tilde\lambda$ is a certificate for the lower bound $L_i := a_i^{\top} y(d,\ell)- \gamma_i(d,\ell) > u_i$.  Hence $A \tilde\lambda_i = -a_i$, $\tilde\lambda_i \geq 0$, and $-u^{\top} \tilde\lambda_i\geq L_i > u_i$. It follows
that $A \bar\lambda_i = 0$, $\bar\lambda_i \geq 0$, and $-u^{\top} \bar\lambda_i \geq L_i - u_i > 0$, and so $\bar\lambda_i$ is a type-\1 certificate of infeasibility.
\end{proof}

The following proposition ties together ellipsoids and slab radii, the condition measure $\tee$, updates of certificates of lower bounds, and type-\1 certificates of infeasibility.  


\begin{Prop} \label{r: cert_const} 
Under the set-up of Proposition \ref{p: cert_const}, let $i := \argmax_{h \in [m]} ( a_h^{\top} y(d,\ell) - u_h )$, and let $L_i$ and $\tilde\lambda_i$ be as defined therein.  If $\gamma_i(d,\ell) < \tau(A,u)$, then $\bar\lambda_i := \tilde\lambda_i + e_i$ is a type-\1 certificate of infeasibility.  
\end{Prop}

\begin{proof}  
If $\mathcal{P}$ were nonempty, then a ball of radius $\tau(A,u)$ would be contained in a slab of radius $\gamma_i(d,\ell)$, so that
$\tau(A,u) \leq \gamma_i(d,\ell)$. Hence $\mathcal{P} = \emptyset$ and so from the definition of $\tau(A,u)$ it follows that 
$$a_i^{\top}y(d,\ell) - u_i \geq \tau(A,u) > \gamma_i(d,\ell) = a_i^{\top}y(d,\ell) - L_i \ , $$ 
where the first inequality is from the definition of $\tau(A,u)$ in the case when $\mathcal{P} = \emptyset$, the second inequality is by supposition, and the equality is from the definition of $L_i$ in Proposition \ref{p: cert_const}.  It then follows that $L_i > u_i$. Hence the desired result follows from Corollary \ref{heatwave}. 
\end{proof}

Propositions \ref{p: cert_const} and \ref{r: cert_const} are premised on $E(d,\ell)$ having positive volume, i.e., $f(d,\ell) >0$.  If $f(d,\ell) \le 0$, then either $f(d,\ell) <0$ or $f(d,\ell) = 0$. If $f(d,\ell) <  0$, then $E(d, \ell) = \emptyset$ from \eqref{tgif}, which implies $\cp = \emptyset$, and if $f(d,\ell)=0$, then \eqref{tgif} implies that either $\cp = \{y(d, \ell)\}$ (which is easy to verify by checking if $A\adj y(d,\ell) \le u$) or $\cp = \emptyset$.  Below we present Procedure \ref{a: cert_for_f_leq_0}, which accomplishes the task of constructing a type-\1 certificate of infeasibility when $f(d,\ell) \leq 0$ and $A\adj y(d,\ell) \not\le u$, i.e., constructing a type-\1 certificate of infeasibility from a type-\2 certificate of infeasibility.  As will be proved, Procedure \ref{a: cert_for_f_leq_0} is well-defined, in that for each step until termination, the specified quantities exist (for example, the scalar $\beta$ in Step \ref{s: beta} exists). Steps \ref{s: beta} through \ref{s: l_dec2} of Procedure \ref{a: cert_for_f_leq_0} decrease the $i$-th and $j$-th entries of $\ell$ such that the updated parameterized ellipsoid $E(d,\ell)$ has positive volume and the condition $a_k^{\top} y(d,\ell) - \gamma_k(d,\ell) > u_k$ holds for some index $k \in [m]$.  Note that in Step \ref{s: l_dec2} it holds that $\Lambda$ is still a certificate for the updated lower bounds $\ell$ because the components of the updated $\ell$ have either been decreased or remain the same.  Steps \ref{s: d_rescale} through \ref{s: cert2} use Proposition \ref{p: cert_const} to construct a certificate $\tilde{\lambda}_k$ for the lower bound $L_k := a_k^{\top} y(d,\ell) - \gamma_k(d,\ell)$ that satisfies $L_k > u_k$. In Step \ref{s: cert3}, we return $\bar\lambda_k := \tilde{\lambda}_k + e_k$, which by Remark \ref{p: cert_inf_const-1} is a type-\1 certificate of infeasibility.  

\floatname{algorithm}{Procedure}
\begin{algorithm}[H]  
\caption{Constructing a type-\1 certificate of infeasibility from a type-\2 certificate of infeasibility} \label{a: cert_for_f_leq_0}
\begin{algorithmic}[1] 
\INPUT{$d >0$, certified lower bounds $\ell$ with certificate matrix $\Lambda$, satisfying $f(d,\ell) \leq 0$ and $A\adj y(d,\ell) \not\le u$} 
\medskip
\State If $\ell \not\le u$, select an index $j$ for which $\ell_j > u_j$ and Return $\bar\lambda_j := \lambda_j + e_j$ and Stop.  \label{kanab}
\State Select any index $i \in [m]$, and compute $\beta \ge 0$ such that $f(d,\ell-\beta e_i) =0$. \label{s: beta} 
\State $\ell \leftarrow \ell - \beta e_i$. \label{s: l_dec1} 
\State Compute an index $j \in [m]$ for which $a_j^{\top}y(d,\ell) \leq u_j$. \label{s: j}  
\State Compute an index $k \in [m]$ for which $a_k^{\top} y(d, \ell) > u_k$. \label{s: k} 
\State Compute $\varepsilon > 0$ such that $f(d,\ell - \varepsilon e_j) > 0$ and $a_k^{\top} y(d, \ell - \varepsilon e_j) - \gamma_k(d, \ell -\varepsilon e_j) > u_k$.\label{s: eps} 
\State $\ell \leftarrow \ell - \varepsilon e_j$. \label{s: l_dec2}
\State $d \leftarrow \frac{1}{f(d,\ell)} d$. \label{s: d_rescale} 
\State $\hat{\lambda}_k \leftarrow \gamma_k(d,\ell) D t(d,\ell) - DA^{\top} B(d)^{-1} a_k$. \label{s: cert1} 
\State $\tilde{\lambda}_k \leftarrow \Lambda \hat{\lambda}_k^- + \hat{\lambda}_k^+$. \label{s: cert2} 
\State Return $\bar\lambda_k := \tilde{\lambda}_k+ e_k$ and Stop. \label{s: cert3} 
\end{algorithmic}
\end{algorithm} 

Proposition \ref{p: alg_correct_f_leq_0} below establishes the correctness of Procedure \ref{a: cert_for_f_leq_0}.  The proof of Proposition \ref{p: alg_correct_f_leq_0} in Appendix \ref{app-119} also indicates how to efficiently implement Steps \ref{s: beta} and \ref{s: eps} of Procedure \ref{a: cert_for_f_leq_0} using the mechanics of the quadratic formula. 

\begin{Prop} \label{p: alg_correct_f_leq_0} 
The output of Procedure \ref{a: cert_for_f_leq_0} is a type-\1 certificate of infeasibility.  \qed
\end{Prop}

\section{Updating the ellipsoid $E(d,\ell)$} \label{s: update} 
Let $d > 0$, let $\ell$ be a certified lower bound for $\poi$ with certificate $\Lambda$, and suppose that the ellipsoid $E(d,\ell)$ has positive volume. Also suppose that $d$ has been scaled so that $f(d,\ell) =1$. In this section, we discuss a procedure for updating the ellipsoid $E(d,\ell)$. We will assume that the center violates some constraint, i.e., the condition $a_j^{\top}y(d,\ell) > u_j$ holds for some $j \in [m]$. Otherwise, $y(d,\ell)$ is feasible, and 
so we would have 
 no reason to construct a new ellipsoid. In the same spirit, we will also assume that the condition $a_j^{\top} y(d,\ell) - u_j \leq \gamma_j(d,\ell)$ holds because if it does not, then we have a type-\3 certificate of infeasibility and can construct a type-\1 certificate of infeasibility (see Corollary \ref{heatwave}), and 
 again we would
  have no reason to construct a new ellipsoid. 

We update the ellipsoid $E(d,\ell)$ by updating its parameters $d$ and $\ell$, and it will be convenient to write the update procedure in five elementary 
steps. At the end of this section, we will provide some motivation for these steps.

\floatname{algorithm}{Procedure}
\begin{algorithm}[H] 
\caption{Updating ellipsoid $E(d,\ell)$ by updating its parameters $d$ and $\ell$}   \label{a: up_ellipse} 
\begin{algorithmic}[1] 
\INPUT{$d >0$ and certified lower bound $\ell$ with certificate matrix $\Lambda$ satisfying $f(d,\ell) = 1$, and $j \in [m]$ such that $0<a_j^{\top} y(d,\ell)-u_j \leq \gamma_j(d,\ell)$.}
\medskip
\State $\hat{\ell} \leftarrow \ell - \frac{2(t_j(d,\ell)-v_j(\ell))}{d_j \gamma_j(d,\ell)^2} e_j.$ \label{upelps1}
\State Compute $y(d,\hat{\ell})$.  If $A^{\top} y(d,\hat{\ell}) \leq u$, then Return $y(d,\hat{\ell})$ as a solution to $\poi$ and Stop. \label{upelps2}  
\State Compute $f(d,\hat{\ell})$.  If $f(d,\hat{\ell}) \le 0$, then call Procedure \ref{a: cert_for_f_leq_0} and Stop. \label{upelps3}
\State $d \leftarrow \frac{1}{f(d, \hat{\ell})} d.$ \label{upelps4}  
\State $\tilde{\ell} \leftarrow \hat{\ell} + \frac{2(2v_j(d, \hat{\ell}) - \gamma_j(d,\hat{\ell}))}{(m-1) d_j \gamma_j(d,\hat{\ell})^2 + 2} e_j.$ \label{upelps5} 
\State $\tilde{d} \leftarrow d + \frac{2}{m-1} \frac{1}{\gamma_j(d,\hat{\ell})^2} e_j.$ \label{upelps6} 
\State $\tilde{d} \leftarrow \frac{1}{f(\tilde{d},\tilde{\ell})} \tilde{d}.$ \label{upelps7} 
\State Return $\tilde{d}$ and $\tilde{\ell}$, and Stop. \label{upelps8}
\end{algorithmic}
\end{algorithm} 

Below we establish several properties of this procedure. Proofs of the results are in Appendix \ref{biden}.

In Step \ref{upelps1} of Procedure \ref{a: up_ellipse}, we update the $j$-th coordinate of $\ell$ to obtain 
\begin{equation} \label{e: update1} 
\ell^{(1)} := \ell - \frac{2(t_j(d,\ell)-v_j(\ell))}{d_j \gamma_j(d,\ell)^2} e_j \ .
\end{equation} 
Note that the numerator above is $a_j^{\top} y(d,\ell) - u_j > 0$, so the $j$th lower bound is decreased. The effect is that the center of the
new ellipsoid $E(d, \ell^{(1)}) $ will satisfy the $j$th inequality at equality, see (\ref{e: shift}) below. While this may hurt the volume reduction achieved, it is
helpful in our analysis of the infeasible case.  Lemma \ref{l: ss} establishes a few properties of this update.

\begin{Lem} \label{l: ss} 
Let $d > 0$ and let $\ell$ be a certified lower bound for $\poi$ with certificate matrix $\Lambda$, and suppose that $f(d,\ell) = 1$.  Also suppose that some $j \in [m]$ satisfies $0<a_j^{\top} y(d,\ell)-u_j \leq \gamma_j(d,\ell)$. Let $\ell^{(1)}$ be defined as in \eqref{e: update1}.  Then $\ell^{(1)}$ is a lower bound for $\poi$ with certificate matrix $\Lambda$, and the following hold:
\begin{align}
 a_j^{\top} y(d,\ell^{(1)}) & = u_j \ , \ \mbox{and}\label{e: shift} \\ 
 f(d, \ell^{(1)}) & =  1 - \left( \frac{a_j^{\top} y(d,\ell) - u_j}{\gamma_j(d,\ell)} \right)^2 \ < \ 1 \ . \label{e: shrink}
\end{align} \qed
\end{Lem} 

From Lemma \ref{l: ss} and the suppositions that $0 < a_j^{\top}y(d,\ell) - u_j \leq \gamma_j(d,\ell)$, it holds that $f(d,\ell^{(1)}) \geq 0$. If $f(d,\ell^{(1)})=0$, then either ${\cal P} = \{y(d, \ell)\}$ in which case in Step \ref{upelps2} we return $y(d,\ell^{(1)})$ as a solution to $\poi$, or $A^{\top}y(d,\ell) \not \le u$ in which case in Step \ref{upelps3} we call Procedure \ref{a: cert_for_f_leq_0} to construct a type-\1 certificate of infeasibility, as discussed and proved in Proposition \ref{p: alg_correct_f_leq_0} of Section \ref{s: updating_certificates}. 

In Steps \ref{upelps4}-\ref{upelps7}, we compute updates 
\begin{align}
& d^{(1)} =  \frac{1}{f(d,\ell^{(1)})} d \ , \label{e: update2} \\  
& \ell^{(2)} = \ell^{(1)} + \frac{2(2v_j(\ell^{(1)}) - \gamma_j(d^{(1)}, \ell^{(1)}))}{(m-1) d_j^{(1)} \gamma_j(d^{(1)},\ell^{(1)})^2 + 2} e_j \ , \label{e: update3} \\
& d^{(2)} = d^{(1)} + \frac{2}{m-1} \frac{1}{\gamma_j(d^{(1)} ,\ell^{(1)})^2} e_j \ . \label{e: update4} \\
& d^{(3)} = \frac{1}{f(d^{(2)},\ell^{(2)})} d^{(2)} \ , \label{e: update5} 
\end{align} 
Lemma \ref{l: construct} establishes a few properties of these updates: 

\begin{Lem} \label{l: construct} 
Let $d > 0$, and let $\ell$ be a certified lower bound for $\poi$ with certificate matrix $\Lambda$, and suppose that $f(d,\ell) = 1$. Also suppose that some $j \in [m]$ satisfies $0<a_j^{\top} y(d,\ell)-u_j \leq \gamma_j(d,\ell)$, and suppose in addition that $\lambda_j$ is a certificate for the lower bound $L_j:= a_j^{\top} y(d,\ell) - \gamma_j(d,\ell)$.  Let $\ell^{(1)}$ be defined as in (\ref{e: update1}), and suppose that $f(d,\ell^{(1)}) > 0$. Let $d^{(1)}$, $\ell^{(2)}$, $d^{(2)}$, and $d^{(3)}$ be defined as in (\ref{e: update2}), (\ref{e: update3}), (\ref{e: update4}), and (\ref{e: update5}),  respectively.  Then 
\begin{enumerate}
\item[(a)]
 $\ell_j^{(2)} \leq \max \left \{ \ell_j, L_j \right \}$, 
and hence $\ell^{(2)}$ is a certified lower bound for $\poi$ with certificate matrix $\Lambda$, 
\item[(b)]
$\gamma_j(d^{(1)},\ell^{(1)}) > 0$,
and hence (\ref{e: update3}) and (\ref{e: update4}) are well-defined, and 
\item[(c)]
 it holds that $ \ \displaystyle d^{(3)} = \displaystyle\frac{m^2-1}{m^2} \left( d^{(1)} + \frac{2}{m-1} \displaystyle\frac{1}{\gamma_j(d^{(1)},\ell^{(1)})^2} e_j \right)$. \qed
\end{enumerate}
\end{Lem}  

\begin{Rem} \label{r: diff1} 
It is possible to show that the ellipsoid $E(d^{(3)}, \ell^{(2)})$ contains the half ellipsoid described by the intersection of the ellipsoid $E(d^{(1)},\ell^{(1)})$ and the half-space $\{ x \in \mathbb{R}^n : a^{\top}_j  x \leq u_j \} = \{ x \in \mathbb{R}^n : a^{\top}_j  x \leq a^{\top}_j  y(d^{(1)},\ell^{(1)}) \}$. The ellipsoid $E(d^{(3)}, \ell^{(2)})$ is not the minimum volume ellipsoid containing the half ellipsoid, but it would be if we substituted $n$ for $m$ in updates 
(\ref{e: update3}) 
and (\ref{e: update4}). We use $m$ instead of $n$ in order to establish a convergence guarantee for our algorithm in the setting in which $\poi$ is infeasible; we will clarify and elaborate on this idea in Section \ref{s: infeasible}. 
\end{Rem} 

So far, we have separately studied the first step and the last four steps of the update procedure. Theorem \ref{t: ssc} below provides a more unified perspective on the update procedure:  

\begin{Th} \label{t: ssc} 
Let $d > 0$ and $\ell$ be a certified lower bound for $\poi$ with certificate matrix $\Lambda$, and suppose that $f(d,\ell) = 1$. Also let $j \in [m]$ be given and suppose that $0<a_j^{\top} y(d,\ell)-u_j \leq \gamma_j(d,\ell)$. Let $\ell^{(1)}$ be defined as in (\ref{e: update1}), and suppose that $f(d,\ell^{(1)}) > 0$.  Let $d^{(1)}$, $\ell^{(2)}$, $d^{(2)}$, and $d^{(3)}$ be defined as in (\ref{e: update2}), (\ref{e: update3}), (\ref{e: update4}), and (\ref{e: update5}) respectively. Then 
$$d^{(3)} = \alpha(d,\ell) \left( d + \frac{2}{m-1} \frac{1}{\gamma_j(d,\ell)^2} e_j \right) \ , $$
 where $\alpha(d,\ell) > \frac{m^2-1}{m^2}$. \qed
\end{Th} 

\begin{Rem}\label{r: motiv}
For those familiar with the ellipsoid algorithm and the Burrell-Todd representation, we 
now give some motivation for the steps in our update procedure. As we will describe in our convergence analysis, our aim is to guarantee
some volume reduction, but at the same time achieve progress in the case of infeasibility. For this, we need to maintain the Burrell-Todd
representation, but it appears we cannot push for aggressive volume reduction. In the original ellipsoid algorithm, the quadratic inequality
defining the new ellipsoid is the sum of that defining the old ellipsoid, say
\[
(x - y)^{\top} B^{-1} (x - y) \leq 1, 
\]
and the multiple
\[
\frac{2}{(n-1)a_j^{\top}Ba_j}
\]
of the quadratic inequality
\[
(a_j^{\top}x - a_j^{\top}y) (a_j^{\top}x - a_j^{\top}y - (a_j^{\top}Ba_j)^{1/2}) \leq 0
\]
(see \cite{tomv}).
However, since the old inequality was not in Burrell-Todd form, nor is the new one. In deep cut and two-sided cut variants, the old ellipsoid
is given by a Burrell-Todd representation. After choosing the constraint $j$, the lower bound $\ell_j$ is possibly updated, and if so, the old
ellipsoid is also updated before again taking a combination of quadratic inequalities to define the new ellipsoid. In our version, this cannot be done
without jeopardizing the analysis. Therefore, the first step is decreasing the lower bound $\ell_j$ to $\ell_j^{(1)}$ so that the new center lies on the
constraint $a_j^{\top} x = u_j$. Then the quadratic inequality defining the new ellipsoid is the sum of that defining the intermediate ellipsoid,
\[
(x - y(d^{(1)},\ell^{(1)}))^{\top} (A D^{(1)} A^{\top}) (x - y(d^{(1)},\ell^{(1)}) \leq 1 \ ,
\]
and the multiple
\[
\frac{2}{(m-1)\gamma_j(d^{(1)},\ell^{(1)})^2}
\]
of the quadratic inequality
\[
(a_j^{\top} x - u_j) (a_j^{\top} x - L_j^{(1)}) \leq 0 \ ,
\]
where $L_j^{(1)} := a_j^\top y(d^{(1)},\ell^{(1)}) - \gamma_j(d^{(1)},\ell^{(1)})$.
Note that, due to the common factor $(a_j^{\top}x - u_j)$, 
the terms $d_j^{(1)} (a_j^{\top} x - u_j) (a_j^{\top} x - \ell_j^{(1)}) $ from the first inequality and
\[
\frac{2}{(m-1)\gamma_j(d^{(1)},\ell^{(1)})^2} (a_j^{\top} x - u_j) (a_j^{\top} x - L_j^{(1)})
\]
can be combined, and this leads to the updates for $\ell^{(2)}$ and $d^{(2)}$ and the new Burrell-Todd representation.
\end{Rem}




\section{Oblivious Ellipsoid Algorithm} \label{s: alg} 
A formal description of our oblivious ellipsoid algorithm (OEA) is presented in Algorithm \ref{alg: 1}; the description is essentially a more formal and detailed version of the schematic version of OEA presented in Algorithm \ref{a: schem}.  
\floatname{algorithm}{Algorithm}
\begin{algorithm}[H] 
\caption{Oblivious Ellipsoid Algorithm (OEA)}   \label{alg: 1} 
\begin{algorithmic}[1] 
\INPUT{data $(A,u)$, certified lower bound $\ell$ for $\poi$ with certificate matrix $\Lambda$, and $d >0$.}
\medskip  
\State Compute $y(d,\ell)$.  If $A^{\top} y(d,\ell) \leq u$, then Return $y(d,\ell)$ as a solution of $\poi$ and Stop. \label{step: feasible} 
\State Compute $f(d,\ell)$.  If $f(d,\ell) \le 0$, then call Procedure \ref{a: cert_for_f_leq_0} and Stop.\label{step: callprocedure}
\State $d \leftarrow \frac{1}{f(d,\ell)} d.$ \label{step: scale1} 
\State Compute most violated constraint:  $j \leftarrow \argmax_{i \in [m]} a_i^{\top} y(d,\ell) - u_i$. \label{step: max_viol}
\State If $\ell_j < L_j := a_j^{\top}y(d,\ell) - \gamma_j(d,\ell)$, then update the certificate for constraint $j$: \label{rain}
\State  \hspace{\algorithmicindent} $\hat{\lambda}_j := \gamma_j(d,\ell) D t(d,\ell) - DA^{\top} B(d)^{-1} a_j.$ \label{raining}
\State  \hspace{\algorithmicindent} $\tilde{\lambda}_j := \Lambda \hat{\lambda}_j^- + \hat{\lambda}_j^+ .$ \label{expensive} 
\State  \hspace{\algorithmicindent} $\lambda_j \leftarrow \tilde\lambda_j. $ \label{morerain}
\State If $L_j > u_j$, then Return type-\1 certificate of infeasibility $\bar\lambda_j := \lambda_j + e_j$ and Stop. \label{step: infeasible} 
\State Update $E(d,\ell)$ by updating ellipsoid parameters $d$ and $\ell$:  \label{step: up_dl} 
\State \hspace{\algorithmicindent} Call Procedure \ref{a: up_ellipse} with input $d$, $\ell$, $\Lambda$ to obtain output $\tilde d, \tilde \ell$.  \label{ellup} 
\State Re-set $(d,\ell) \leftarrow (\tilde d, \tilde \ell)$ and Goto Step \ref{step: feasible}. \label{step: repeat} 
\end{algorithmic}
\end{algorithm} 

Let us briefly consider the steps of Algorithm \ref{alg: 1} that are different from the steps of the schematic Algorithm \ref{a: schem}. In Step \ref{step: callprocedure} when $f(d,\ell) \le 0$, we call Procedure \ref{a: cert_for_f_leq_0} (see Section \ref{s: updating_certificates}) to construct and return a type-\1 certificate of infeasibility. In Steps \ref{rain}-\ref{morerain} we use Proposition \ref{p: cert_const} to update the certificate $\lambda_j$ when we can construct a new certificate that certifies a better lower bound (namely $L_j$). In Step \ref{step: infeasible} when $L_j > u_j$, we return $\bar{\lambda}_j$, which is a type-\1 certificate of infeasibility by Corollary \ref{heatwave}. Lastly, in Step \ref{ellup}, we use Procedure \ref{a: up_ellipse} (of Section \ref{s: update}) to update the ellipsoid by updating its parameters. 

\begin{Rem} \label{notsofast} 
The operations complexity of an iteration of OEA (with appropriate rank-$1$ updates) is $O(m^2)$ because the most expensive computation that can occur in an iteration is computing $\Lambda \hat{\lambda}_j^{-}$ in Step \ref{expensive}. 
\end{Rem} 

\begin{Rem} \label{r: once} 
It turns out that the condition $f(d,\ell) \leq 0$ in Step \ref{step: callprocedure} can only be satisfied at Step \ref{step: callprocedure} during the first iteration of the algorithm.  This is because at later iterations Procedure \ref{a: up_ellipse} in Step \ref{ellup} detects if this condition holds, calls Procedure 2, and then terminates. (And it is straightforward to check that if Procedure \ref{a: up_ellipse} completes a full iteration with output $\tilde{d}$ and $\tilde{\ell}$, then $f(\tilde{d}, \tilde{\ell}) >0$.)  

In a similar spirit, Step \ref{step: scale1} only needs to be implemented during the first iteration of the algorithm because Procedure \ref{a: up_ellipse}, called in Step \ref{ellup}, returns parameters $\tilde{d}$ and $\tilde{\ell}$ that satisfy $f(\tilde{d},\tilde{\ell}) = 1$. 
\end{Rem}

\subsection{Computational Guarantees when $\poi$ is Infeasible} \label{s: infeasible} 
In this subsection we examine the computational complexity of Algorithm \ref{alg: 1} in the case when $\poi$ is infeasible.  We start with the following elementary proposition that bounds the slab radii in terms of the (normalized) components of $d$. 
(Proofs of the results of Section \ref{s: infeasible} appear in Appendix \ref{finnigan}.)

\begin{Prop} \label{l: 1} 
Let $d \in \mathbb{R}^m_{++}$ and $\ell \in \mathbb{R}^m$ such that $f(d,\ell) > 0$. For all $i \in [m]$ it holds
 that

$$\gamma_i(d,\ell) \le\left( \frac{d_i}{f(d,\ell)}  \right)^{-\frac{1}{2}} \ . $$ \qed
\end{Prop} 

It then follows from Proposition \ref{r: cert_const} and Proposition \ref{l: 1} that we can construct a type-\1 certificate of infeasibility if the entries of the normalized iterate $\frac{1}{f(d,\ell)} d$ eventually become large enough so that they satisfy
$$\left( \frac{1}{f(d,\ell)} d_i \right)^{-\frac{1}{2}} < \tau(A,u) \ \ \mbox{for~all~}i \in [m] \ . $$

In order to prove that this condition will eventually hold, we first introduce the following potential function $\phi(d,\ell)$: $$\phi(d,\ell) := \prod_{i=1}^m \max \left \{  \left( \frac{1}{f(d,\ell)} d_i \right)^{-\frac{1}{2}} , \frac{m}{m+1} \tau(A,u) \right \} \ , $$
and we will show in this subsection that this potential function sufficiently decreases over the iterations in the case when $\poi$ is infeasible.   For notational convenience, define 
$$\mu_i(d,\ell) := \max \left \{  \left( \frac{1}{f(d,\ell)} d_i \right)^{-\frac{1}{2}} ,  \frac{m}{m+1} \tau(A,u) \right \} \ , $$
and therefore $\phi(d,\ell) = \prod_{i=1}^m \mu_i(d,\ell)$.  Note that $\phi(d,\ell)$ is bounded from below, namely $\phi(d,\ell) \geq \left( \frac{m}{m+1} \tau(A,u) \right)^m$.  Lemma \ref{t: main_t} below states that after updating $d$ and $\ell$ in Procedure \ref{a: up_ellipse}, $\phi(d,\ell)$ decreases by at least the multiplicative factor $e^{-\frac{1}{2(m+1)}}$ . 

\begin{Lem}[Potential function decrease] \label{t: main_t} 
Let $d > 0$ and $\ell \in \mathbb{R}^m$ satisfy $f(d,\ell) > 0$, and similarly let $\tilde{d} > 0$ and $\tilde{\ell} \in \mathbb{R}^m$ satisfy $f(\tilde{d},\tilde{\ell}) > 0$. Let $j \in [m]$ be given, and suppose that $d$, $\ell$, $\tilde d$, $\tilde \ell$ satisfy:  
$$\frac{1}{f(\tilde{d},\tilde{\ell})} \tilde{d} = \alpha \left(\frac{1}{f(d,\ell)} d+ \frac{2}{m-1} \frac{1}{\gamma_j(d,\ell)^2} e_j \right) \ , $$
for a scalar $\alpha \geq \frac{m^2-1}{m^2}$.  If $\left( \frac{d_j}{f(d,\ell)}  \right)^{-\frac{1}{2}} \geq \tau(A,u)$, then 
$$\phi(\tilde{d}, \tilde{\ell}) \leq e^{-\frac{1}{2(m+1)}} \phi(d,\ell) \ .   $$\qed
\end{Lem}

With Lemma \ref{t: main_t} in hand, we now state and prove our main computational guarantee for Algorithm \ref{alg: 1} in the case when $\poi$ is infeasible.

\begin{Th} \label{t: infeas_guar} 
Let $\ell \in \mathbb{R}^m$ be certified lower bounds for $\poi$ with certificate matrix $\Lambda$.  Let $d:=e \in \mathbb{R}^m$.  If $\poi$ is infeasible, Algorithm \ref{alg: 1} with input $A$, $u$, $\ell$, $\Lambda$, and $d$ will stop and return a type-\1 certificate of infeasibility in at most $$\left\lfloor 2m(m+1) \ln \left( \frac{m+1}{2m} \frac{\lVert u - \ell \rVert}{\tau(A,u)} \right) \right\rfloor$$ iterations. \qed
\end{Th} 

\begin{proof}  In the notation of the theorem $d$ and $\ell$ are the initial values used as input to Algorithm \ref{alg: 1}.  First note that if $f(d,\ell) \le0$, then it follows from Proposition \ref{p: alg_correct_f_leq_0} that Step \ref{step: callprocedure} of Algorithm \ref{alg: 1} will return a certificate of infeasibility of $\poi$ at the very first iteration.  Also, if $\sqrt{f(e,\ell)} < \tee$, then it follows from Proposition \ref{l: 1} that $\gamma_i(e,\ell) < \tee$ for all $i \in [m]$, whereby for the index $j$ in Step \ref{step: max_viol} it holds that ${a_j}\adj y(d,\ell) - u_j \ge \tee > \gamma_j(e,\ell)$ which implies that $L_j > u_j$ in Step \ref{step: infeasible} of Algorithm \ref{alg: 1}, and so it follows from Corollary \ref{heatwave} that Algorithm \ref{alg: 1} will return a certificate of infeasibility of $\poi$ at Step \ref{step: infeasible} of the very first iteration.  We therefore suppose for the rest of the proof that $f(e, \ell) > 0$ and $\sqrt{f(e,\ell)} \ge \tee$.  

From the definition of the potential function, it therefore holds for the initial values of $d=e$ and $\ell$ that $\phi(e,\ell) = \Pi_{i=1}^m \sqrt{f(e,\ell)}$.  Notice that $f(e, \ell) = v(\ell)\adj I v(\ell) - t(e,\ell)\adj I t(d,\ell) \le v(\ell)\adj I v(\ell) = (\tfrac{1}{2}\|u - \ell\|)^2$, whereby $\phi(e,\ell) \le (\tfrac{1}{2}\|u - \ell\|)^m$.  

Suppose that Algorithm \ref{alg: 1} has completed $k$ iterations, and let $\hat d$ and $\hat \ell$ denote the values of $d$ and $\ell$ upon completion of iteration $k$.  It then follows from Lemma \ref{t: main_t} that

$$\left( \frac{m}{m+1} \tau(A,u) \right)^m \le \phi(\hat d, \hat \ell) \le e^{-\frac{k}{2(m+1)}} \phi(d,\ell) \le e^{-\frac{k}{2(m+1)}}(\tfrac{1}{2}\|u - \ell\|)^m \ , $$ where the first inequality uses the absolute lower bound on $\phi(\cdot, \cdot)$ from its definition, and the second inequality uses Lemma \ref{t: main_t}.  Taking logarithms of both sides and rearranging terms yields the inequality 
$ k \le  2m(m+1) \ln \left( \frac{m+1}{2m} \frac{\lVert u - \ell \rVert}{\tau(A,u)} \right) $ which proves the result.\end{proof}

Corollary \ref{c: infeas_box_guar} specializes Theorem \ref{t: infeas_guar} to instances of linear inequality systems with box constraints $(P_B)$ from Section \ref{ss: valid_lbs}. The corollary follows immediately from Theorem \ref{t: infeas_guar} and inequality (\ref{e: box_bound}). 

\begin{Cor} \label{c: infeas_box_guar} 
Consider the linear inequality system with box constraints $(P_B)$, and let $A$, $u$, $\ell$, and $\Lambda$ be defined as in \eqref{e: l_box}-\eqref{e: lam3_box}.  Let $d:=e \in \mathbb{R}^m$.  If $(P_B)$ is infeasible, Algorithm \ref{alg: 1} with input $A$, $u$, $\ell$, $\Lambda$, and $d$ will stop and return a type-\1 certificate of infeasibility in at most $$\left\lfloor 2m(m+1) \ln \left( \frac{(m+1)(\sqrt{\hat m +2})}{2m} \frac{\lVert \bar b - \underline b \rVert}{\tau(A,u)} \right) \right\rfloor$$ iterations. 

\end{Cor} \qed

\subsection{Computational Guarantees when $\poi$ is Feasible} \label{s: feasible} 
In this subsection we examine the computational complexity of Algorithm \ref{alg: 1} in the case when $\poi$ is feasible.  Our analysis is in some sense standard, in that we show that upon updating the values of $d$ and $\ell$ in Procedure \ref{a: up_ellipse}, the volume of the newly updated ellipsoid $E(d,\ell)$ decreases by a sufficient amount.  For $d \in \mathbb{R}^m_{++}$ and $\ell \in \mathbb{R}^m$ satisfying $f(d,\ell) > 0$, the (relative) volume of $E(d,\ell)$ is:
$$\vol E(d,\ell)  := \frac{(f(d,\ell))^{\tfrac{n}{2}}}{\sqrt{\det ADA^{\top}}} \  $$
(relative in that it ignores the dimensional constant $c_n = \frac{\pi^{(n/2)}}{\Gamma(n/2+1)}$).  Lemma \ref{l: volume_dec} below states that after updating $d$ and $\ell$ in Procedure \ref{a: up_ellipse}, the volume of the ellipsoid $E(d,\ell)$ decreases by at least the multiplicative factor $e^{-\frac{1}{2(m+1)}}$.

\begin{Lem}[Volume decrease] \label{l: volume_dec}  
Let $d > 0$ and $\ell \in \mathbb{R}^m$ satisfy $f(d,\ell) > 0 $, and similarly let  $\tilde{d} > 0$ and $\tilde{\ell} \in \mathbb{R}^m$ satisfy $f(\tilde{d},\tilde{\ell}) > 0$.  Let $j \in [m]$ be given, and suppose that $d$, $\ell$, $\tilde d$, $\tilde \ell$ satisfy: $$\frac{1}{f(\tilde{d},\tilde{\ell})} \tilde{d} = \alpha \left(\frac{1}{f(d,\ell)} d+ \frac{2}{m-1} \frac{1}{\gamma_j(d,\ell)^2} e_j \right) \ , $$ for a scalar $\alpha \geq \frac{m^2-1}{m^2}$. Then
$$\vol E(\tilde{d},\tilde{\ell}) \leq e^{-\frac{1}{2(m+1)}}  \vol E(d,\ell) \ . $$ 
\end{Lem} \qed

With Lemma \ref{l: volume_dec}  in hand, we now state and prove our main computational guarantee for Algorithm \ref{alg: 1} in the case when $\poi$ is feasible.  Note that the theorem uses the condition number $\cc$, which was introduced in \eqref{fridays} and measures the distance to unboundedness as discussed earlier.

\begin{Th} \label{t: feas_guar} 
Let $\ell \in \mathbb{R}^m$ be certified lower bounds for $\poi$ with certificate matrix $\Lambda$.  Let $d:=e \in \mathbb{R}^m$.  If $\poi$ is feasible, Algorithm \ref{alg: 1} with input $A$, $u$, $\ell$, $\Lambda$, and $d$ will stop and return a feasible solution of $\poi$ in at most $$\left\lfloor 2n(m+1) \ln \left( \frac{\lVert u - \ell \rVert}{2 \cc \tau(A,u)} \right) \right\rfloor$$ iterations. \qed
\end{Th} 

\begin{proof}  In the notation of the theorem $d=e$ and $\ell$ are the initial values used as input to Algorithm \ref{alg: 1}.  Let us first bound the volume of the initial ellipsoid $E(d,\ell)$.  From the bound on $f(e,\ell)$ in the proof of Theorem \ref{t: infeas_guar} we have:
$$\vol E(d,\ell) = \frac{(f(e,\ell))^{\tfrac{n}{2}}}{\sqrt{\det AA^{\top}}}  \le \frac{(\tfrac{1}{2}\|u - \ell\|)^n}{\cc^n} = \left(\frac{\|u - \ell\|)}{2\cc}\right)^n \ , $$
where the bound in the denominator above uses Proposition \ref{evie}.  Suppose that Algorithm \ref{alg: 1} has completed $k$ iterations, and let $\hat d$ and $\hat \ell$ denote the values of $d$ and $\ell$ upon completion of iteration $k$.  Next notice that  $E(\hat d, \hat \ell) \supset \cp \supset B(c,\tee)$ for some $c \in \cp$ where the second inclusion follows from Proposition \ref{taulist}.  Therefore a lower bound on $\vol E(\hat d, \hat \ell)$ is $\tee^n$.  It then follows from Lemma \ref{l: volume_dec} that

$$\tee^n \le \vol E(\hat d, \hat \ell) \le e^{-\frac{k}{2(m+1)}} \vol E(d,  \ell) \le e^{-\frac{k}{2(m+1)}}\left(\frac{\|u - \ell\|)}{2\cc}\right)^n \ , $$ where the second inequality uses Lemma \ref{t: main_t} and the third inequality uses the upper bound on $\vol E(d,\ell)$.  Taking logarithms of both sides and rearranging terms yields the inequality 
$ k \le  2n(m+1)  \ln \left( \frac{\lVert u - \ell \rVert}{2 \cc \tau(A,u)} \right) $ which proves the result.\end{proof}

Corollary \ref{c: feas_box_guar} specializes Theorem \ref{t: feas_guar} to instances of linear inequality systems with box constraints $(P_B)$ from Section \ref{ss: valid_lbs}. The corollary follows from Theorem \ref{t: feas_guar}, inequality (\ref{e: box_bound}), and the fact that if $A^{\top} x \leq u$ contains box constraints, then $\det AA^{\top} > 1$ (and so $\rho(A)$ vanishes in the guarantee).  

\begin{Cor} \label{c: feas_box_guar} 
Consider the linear inequality system with box constraints $(P_B)$, and let $A$, $u$, $\ell$, and $\Lambda$ be defined as in \eqref{e: l_box}-\eqref{e: lam3_box}.  Let $d:=e \in \mathbb{R}^m$.  If $(P_B)$ is feasible, Algorithm \ref{alg: 1} with input $A$, $u$, $\ell$, $\Lambda$, and $d$ will stop and return a feasible solution of $(P_B)$ in at most $$\left\lfloor 2n(m+1) \ln \left( \frac{\sqrt{\hat m +2}\lVert \bar b -\underline b \rVert}{2 \tau(A,u)} \right) \right\rfloor$$ iterations. 

\end{Cor} \qed\medskip

We conclude this section by pointing to the computational complexity of OEA.  Remark \ref{notsofast} states that the operations complexity of an iteration of OEA is $O(m^2)$ operations.  Combining this with the iteration complexity of Theorem \ref{t: infeas_guar} and Theorem \ref{t: feas_guar} yields the computational complexity bounds for OEA in the second row of Table \ref{comp}.

\section{Modified Versions of Algorithm OEA} \label{oeas}
In this section we present two modified versions of OEA, which we call OEA-No-Alt and OEA-MM, for Challenges \ref{p2} and \ref{p1}, respectively.

\subsection{OEA-No-Alt} 
OEA-No-Alt is a simpler version of OEA that does not iteratively update the information needed to produce a type-\1 certificate of infeasibility.  The algorithm still proves infeasibility by correctly detecting infeasibility when $\poi$ is infeasible, but it does not produce a solution of $\alt$, hence the notation ``OEA-No-Alt.''  The modified algorithm is based on two rather elementary observations about OEA, as follows.

The first observation concerns the role of the updates of $\Lambda$ in OEA.  Observe that the certificate matrix $\Lambda$ is never used anywhere in the computational rules in OEA nor in the updates of any objects other than $\Lambda$ itself; these updates of $\Lambda$ are pure ``record-keeping'' and their sole purpose is to eventually produce a Type-\1 certificate of infeasibility (a solution of $\alt$) after such infeasibility is detected and the algorithm needs no further iterations.  Hence, if one is not interested in actually computing a solution of $\alt$, any and all updates of $\Lambda$ can be omitted.  By omitting the updates of $\Lambda$ the algorithm will no longer produce a solution of $\alt$ in the case when $\poi$ is infeasible, and hence we denote this simplified version of OEA as OEA-No-Alt.  Nevertheless the updated values of $\Lambda$ exist (but are just not computed).  

A somewhat formal description of OEA-No-Alt is as follows. Instead of calling Procedure \ref{a: cert_for_f_leq_0} in Step \ref{step: callprocedure} when $f(d,\ell) \leq 0$ (implying that $\poi$ is infeasible), OEA-No-Alt simply declares infeasibility and
stops. In Step \ref{rain} there is no update of the certificate for constraint $j$, and Steps \ref{raining}-\ref{morerain} are thus omitted. And instead of returning a type-\1 certificate of infeasibility in Step \ref{step: infeasible} when $L_j > u_j$ (implying $\poi$ is infeasible), OEA-No-Alt simply declares infeasibility and stops. Finally, in Step \ref{ellup}, instead of calling Procedure \ref{a: cert_for_f_leq_0} inside of Procedure \ref{a: up_ellipse} (implying $\poi$ is infeasible), OEA-No-Alt simply declares infeasibility and stops. 

Notice from the above formal description of OEA-No-Alt that the stopping criteria in the case when $\poi$ is infeasible are identical to that in the original OEA.  Hence, in the case when $\poi$ is infeasible, OEA-No-Alt will stop when and only when it detects infeasibility exactly as in the original OEA.

The second observation concerns the operations complexity of an iteration of OEA. The computational complexity of an iteration of OEA is $O(mn)$ operations except for the updates of the certificate matrix $\Lambda$, which are $O(m^2)$ operations.  Therefore, if we eliminate the updates of the matrix $\Lambda$, the operations complexity of an iteration of the resulting algorithm is $O(mn)$ operations.

The above analysis yields the following computational complexity result for OEA-No-Alt.

\begin{Cor}\label{wowie} Let $\ell \in \mathbb{R}^m$ be certified lower bounds for $\poi$ with certificate matrix $\Lambda$.  Let $d:=e \in \mathbb{R}^m$.  If $\poi$ is infeasible, Algorithm OEA-No-Alt with input $A$, $u$, $\ell$, $\Lambda$, and $d$ will correctly detect infeasibility, proving that $\poi$ is infeasible, with the same iteration bound as given in Theorem \ref{t: infeas_guar}.  Therefore the total computational complexity of Algorithm OEA-No-Alt is $O(m^3n \ln \frac{1}{\tau})$ operations. \qed
\end{Cor}

The computational complexity bounds for OEA-No-Alt in the third row of Table \ref{comp} follow directly from the above observations.

\subsection{OEA-MM}\label{covid3}
OEA-MM is very similar to OEA-No-Alt, except that it computes and stores certain information at each iteration that can be used after-the-fact to later construct the final certificate matrix $\Lambda$ that would have been produced by the complete OEA algorithm.  In this way, if $\poi$ is infeasible, the final $\Lambda$ can be constructed after-the-fact and used to produce the solution of $\alt$ exactly as in the complete OEA.  And if $\poi$ is feasible, no certificate matrix is needed and so computing $\Lambda$ is unnecessary.  OEA-MM is based on the following notions.  
\begin{enumerate}
\item Just like OEA-No-Alt, OEA-MM does not iteratively update the certificate matrix $\Lambda$, and in this aspect it is identical to OEA-No-Alt.  By not updating the certificate matrix $\Lambda$ at each iteration, the per-iteration complexity is reduced to $O(mn)$ operations per iteration just like in the algorithm OEA-No-Alt.  
\item However, in the interest of having the capability of computing the solution of $\alt$ after-the-fact that would have been computed by the complete algorithm OEA if $\poi$ is infeasible, OEA-MM computes and stores the information needed to re-construct the certificate of infeasibility that would be computed by the complete algorithm OEA.  For this reason the algorithm has the notation ``-MM'' for more memory.
\item If $\poi$ is infeasible, the information stored at each iteration is then used to construct the type-\1 certificate of infeasibility that the complete algorithm OEA would have computed.
\end{enumerate} 

Before going into the details of algorithm OEA-MM, we first step back and examine certain properties of the complete algorithm OEA under the assumption that $\poi$ is infeasible.  OEA constructs a type-$L$ certificate of infeasibility either in Step \ref{step: infeasible} of Algorithm \ref{alg: 1} or in Step \ref{s: cert3} of Procedure \ref{a: cert_for_f_leq_0} (after being called by Algorithm \ref{alg: 1}).  Let $k$ be the number of iterations of algorithm OEA in which the certificate matrix is updated in Step \ref{expensive} of Algorithm \ref{alg: 1} or Step \ref{s: cert2} of Procedure \ref{a: cert_for_f_leq_0} (an upper bound on $k$ is given in Theorem \ref{t: infeas_guar}).  Let us denote the $i$-th certificate matrix that OEA constructs by $\Lambda^{(i)}$ for $i \in [k]$.  For consistency, we denote the initial given certificate matrix as $\Lambda^{(0)}$.  

OEA updates the previous certificate matrix $\Lambda^{(i-1)}$ to the new certificate matrix $\Lambda^{(i)}$ in Step \ref{expensive} of Algorithm \ref{alg: 1} or Step \ref{s: cert2} of Procedure \ref{a: cert_for_f_leq_0} by first computing the relevant index $j_i := j$ in Step \ref{step: max_viol} of Algorithm \ref{alg: 1} or $j_i := k$ in Step \ref{s: k} of Procedure \ref{a: cert_for_f_leq_0}, along with the vector $\hlambda_{(i)} := \hlambda_j$ in Step \ref{raining} of Algorithm \ref{alg: 1} or $\hlambda_{(i)} := \hlambda_k$ in Step \ref{s: cert1} of Procedure \ref{a: cert_for_f_leq_0}.  Finally, according to Step \ref{expensive} of Algorithm \ref{alg: 1} or Step \ref{s: cert2} of Procedure \ref{a: cert_for_f_leq_0}, we obtain $\Lambda^{(i)}$ by updating $\Lambda^{(i-1)}$ which works out in full matrix form to be:
\begin{equation} \label{covid1} 
\Lambda^{(i)} = \Lambda^{(i-1)}[ I - e_{j_i} e_{j_i}^\top + \hlambda^-_{(j_i)}e_{j_i}^\top] + [\hlambda^+_{(j_i)}e_{j_i}^\top] =  \Lambda^{(i-1)}M_{(i)} + B_{(i)} \ , 
\end{equation} where $$M_{(i)} :=  I - e_{j_i} e_{j_i}^\top + \hlambda^-_{(j_i)}e_{j_i}^\top \ \ \mbox{and} \ \ B_{(i)} := \hlambda^+_{(j_i)}e_{j_i}^\top \ . $$
First notice that given $\Lambda^{(0)}$ and if we have computed and stored the matrix pairs \\
$(M_{(1)},B_{(1)}), \ldots, (M_{(k)}, B_{(k)})$, we can construct the final certificate matrix $\Lambda^{(k)}$ by inductively using \eqref{covid1}, and then construct the type-\1 certificate of infeasibility by the computation $\bar\lambda_{j_k} := \Lambda^{(k)} e_{j_k} + e_{j_k}$.  Next notice that it is sufficient to compute and store the vector-index pairs $(\hlambda_{(1)},j_1), ..., (\hlambda_{(k)},j_k)$ rather than the full matrices $(M_{(1)},B_{(1)}), \ldots, (M_{(k)}, B_{(k)})$ because for each iteration $i$ the information contained in the pair $(\hlambda_{(i)},j_i)$ is sufficient to construct the matrices $(M_{(i)},B_{(i)})$.  We will refer to the sequence $\{(\hlambda_{(i)}, j_i)\}_{i \in [k]}$ as the \emph{certificate-index sequence} of OEA.  

Based on the above discussion, we obtain OEA-MM from algorithm OEA with the following modifications:

\begin{enumerate} 
\item OEA-MM foregoes Steps \ref{expensive} and \ref{morerain} of Algorithm \ref{alg: 1} and Step \ref{s: cert2} Procedure \ref{a: cert_for_f_leq_0}. Accordingly, OEA-MM does not update the certificate matrix $\Lambda$. 

\item After implementing Step \ref{raining} of Algorithm \ref{alg: 1} and Step \ref{s: cert1} of Procedure \ref{a: cert_for_f_leq_0}, OEA-MM stores $\hat{\lambda}_j$ and $j$ as a pair $(\hat{\lambda}_j, j)$ in memory. These pairs comprise the certificate-index sequence $(\lambda_{(1)},j_1), ..., (\lambda_{(k)},j_k)$.

\end{enumerate}

It follows from Theorems \ref{t: main_t} and \ref{t: feas_guar} that the number $k$ of certificate-index pairs that need to be stored by OEA-MM satisfies $k = O(\max\{m^2 \ln(\lVert u - \ell \rVert/\tau(A,u)) , mn \ln(\lVert u - \ell \rVert/( \cc \tau(A,u)) \})$.

Finally, notice that Step \ref{step: infeasible} of Algorithm \ref{alg: 1} or Step \ref{s: cert3} of Procedure \ref{a: cert_for_f_leq_0} are where the type-\1 certificate of infeasibility is computed in the complete OEA.  A naive (and inefficient) way to accomplish the computation in these steps in OEA-MM would be to construct the full final certificate matrix $\Lambda^{(k)}$ by inductively using \eqref{covid1}, and then to construct the type-\1 certificate of infeasibility by the computation $\bar\lambda_{j_k} := \Lambda^{(k)} e_{j_k} + e_{j_k}$.  However, we can take advantage of the inductive recursion in the construction of $\Lambda^{(k)}$ in \eqref{covid1} to instead compute just the type-\1 certificate of infeasibility $\bar\lambda_{j_k}$ via a sequence of $k$ back-solves.  The detailed computation is presented in Procedure \ref{postc} below.  

\floatname{algorithm}{Procedure}
\begin{algorithm}[H] 
\caption{Construction of Type-L Certificate from (Stored) Certificate-Index Sequence}   \label{postc} 
\begin{algorithmic}[1] 
\INPUT{initial certificate matrix $\Lambda^{(0)}$ and certificate-index sequence $\{(\hlambda_{(i)},j_i) \}_{i \in [k]}$ .}
\medskip  
\State Initialize $w^k \leftarrow e_{j_k}$ and $z^k \leftarrow e_{j_k}$ . 
\For{$i = k:1$}
\State $w^{i-1} \leftarrow w^i + (\hlambda_{(i)}^- - e_{j_i})(e_{j_i}^{\top} w^i)$ . \label{a} 
\State $z^{i-1} \leftarrow \hlambda^+_{(i)} e_{j_i}^{\top} w^i + z^i$ . \label{b} 
\EndFor
\State Return $\bar\lambda := \Lambda^{(0)}w^0 + z^0$, and Stop. \label{L0} 
\end{algorithmic}
\end{algorithm} 

\noindent The following proposition establishes the correctness of Procedure \ref{postc}.  The proof of Proposition \ref{covid2} is given in Appendix \ref{vistalives}.

\begin{Prop} \label{covid2} 
The output of Procedure \ref{postc} satisfies $\bar\lambda = \Lambda^{(k)} e_{j_k} + e_{j_k}$ and hence is a type-\1 certificate of infeasibility.  \qed
\end{Prop} 

\noindent Based on this procedure, the third and final modification of OEA is as follows:

\begin{enumerate} 

\item[3.] Once Step \ref{step: infeasible} of Algorithm \ref{alg: 1} or Step \ref{s: cert3} of Procedure \ref{a: cert_for_f_leq_0}  must be executed, OEA-MM instead runs Procedure \ref{postc} to construct a Type-\1 certificate from the initial certificate matrix $\Lambda^{(0)}$ and the certificate-index sequence $\{(\lambda_{(i)},j_i)\}_{i \in [k]}$ stored in memory. 

\end{enumerate}

The computational complexity of implementing Steps \ref{a} and \ref{b} in Procedure \ref{postc} is $O(m)$, and Step \ref{L0} requires $O(m^2)$ operations. From the earlier discussion, the computational complexity of the number of iterations $k$ of Procedure \ref{postc} is $O(m^2)$. Thus the computational complexity of Procedure \ref{postc} is $O(m^3)$. Finally, because OEA-MM foregoes Step \ref{expensive} of Algorithm \ref{alg: 1}, the total computational complexity of OEA-MM is $O(m^3n \ln \frac{1}{\tau})$. 

The above analysis yields the following computational complexity result for OEA-MM.

\begin{Cor}\label{wowow} Let $\ell \in \mathbb{R}^m$ be certified lower bounds for $\poi$ with certificate matrix $\Lambda$.  Let $d:=e \in \mathbb{R}^m$.  If $\poi$ is infeasible, Algorithm OEA-MM with input $A$, $u$, $\ell$, $\Lambda$, and $d$ will stop and return a type-\1 certificate of infeasibility, with the same iteration bound as given in Theorem \ref{t: infeas_guar}.  Therefore the total computational complexity of Algorithm OEA-MM is $O(m^3n \ln \frac{1}{\tau})$ operations. 
\end{Cor}

Finally, the computational complexity bounds for OEA-MM in the fourth row of Table \ref{comp} follow directly from the above observations. 



\appendix
\appendixpage

\section{Reducing the Operation Counts for an Iteration of the Ellipsoid Algorithm} \label{impl_Alt}

The first two subsections of this appendix present two different transformations for implementing the standard ellipsoid algorithm to solve the problem $\alt$, whose feasible region lies in the nullspace of $A$ (which we denote by $\nulla$).  The first transformation preserves Euclidean distances and hence aspects of the original conditioning of $\alt$, and is presented in Section \ref{eric1}.  The second transformation is guaranteed to reduce the operation counts per iteration of the ellipsoid algorithm from $O(mn)$ to $O(np)$ operations per iteration where $p:=m-n$, but the condition measures $\tau(\cdot)$ and $\rho(\cdot)$ are changed by the transformation.  This is presented in Section \ref{eric2}.  Section \ref{hirob} of this appendix shows how the transformation in Section \ref{eric2} can be applied when solving $\poi$ using either the standard ellipsoid algorithm, OEA, OEA-No-Alt, or OEA-MM. This transformation is guaranteed to reduce the operation counts per iteration of these versions by the factor $m/p$ -- but the condition measures $\tau(\cdot)$ and $\rho(\cdot)$ are changed by the transformation.  

\subsection{QR factorization to implement the iterations of the standard ellipsoid algorithm for solving $\alt$}\label{eric1}

This subsection considers a QR factorization approach to implement the standard ellipsoid method for solving $\alt$, that preserves Euclidean distances (and hence key features of problem geometry of $\alt$).  Specifically, we show how to parameterize the nullspace of $A$ (which we denote by $\nulla$) in a way that preserves Euclidean distances in $\alt$.  We initially compute the QR factorization of $A^\top$: we write
\[
A^\top = QR = [Y, Z] \left[ \begin{array}{cc} R_Y \\ 0 \end{array} \right] = Y R_Y \ ,
\]
where $Q$ is an $m \times m$ orthogonal matrix, partitioned into its first $n$ and last $p$ columns, and $R$ is an
$m \times n$ upper triangular matrix partitioned into its first $n$ and last $p$ rows. We then have
 $AZ = R_Y^\top Y^\top Z = 0$, and the columns of $Z$ form an orthonormal basis for the nullspace of $A$
 (recall that Assumption \ref{assu1} implies $A$ has rank $n$, so that its nullspace has dimension $p$). Then
 $\{ \lambda \in \mathbb{R}^m: A \lambda = 0 \} = \{ Z \mu: \mu \in \mathbb{R}^p \}$ and, since $Z$ has orthonormal columns, 
 the distance between two points $\lambda$ in $\nulla$ coincides with the
 distance between their corresponding $\mu$'s. This property preserves the Euclidean geometry of the problem.  
 
 To compute the factorization we calculate a sequence of elementary reflectors of the form $Q_i := I - 2 w_i w_i^\top$,
 where each $w_i$ is a unit vector, to reduce $A^\top$ to upper triangular form column by column, so that
 \[
 Q_n \cdots Q_2 Q_1 A^\top = R \ , \qquad Q = Q_1 Q_2 \cdots Q_n \ .
 \]
 This requires $O(m n^2)$ arithmetic operations. We can next if desired compute $Z$ column by column in 
 $O(mnp)$ operations. We can then apply the ellipsoid method to seek a point satisfying the transformed version of $\alt$ which is 
\begin{equation} \nonumber \label{doi} 
\malt:  \ \ \left\{ \begin{array}{rcl}
Z \mu & \ge& 0 \\
u^{\top}Z \mu& <& 0 \\
\|\mu\| & \le & 1 \ , \end{array} \right.
\end{equation} 
starting with the initial ellipsoid equal to the unit ball, and where we have added the unit ball constraint due to the positive homogeneity (of degree $1$) of the rest of the system.  At every iteration we need to evaluate the constraints at the current center.
If we have computed and stored $Z$ explicitly, this requires $O(mp)$ operations; alternatively we can augment $\mu$ with $n$ leading zeros and apply the $Q_i$'s sequentially to process the first two constraints of $\malt$ in $O(mn)$ operations. We compute the normal vector of a violated constraint by inspecting
$Z$, or alternatively by computing $Z^\top e_j$ or $Z^\top u$ in $O(mn)$ operations.  We then proceed to update the inverse shape matrix (or a factorization of
the shape matrix) in $O(p^2)$ operations. Choosing the better of the above alternatives, we see that each iteration of the standard ellipsoid algorithm applied to solve $\malt$ uses 
$O(mp)$ operations.  (Although $Q$ differs from the identity by a matrix of rank $n$, we do not see
how to reduce the operation count for updating the shape matrix when $n \ll m$.) 

Let us also see how to implement the above computational steps of the ellipsoid method directly in the space of $\lambda$'s. At a particular iteration, we suppose that we
have an ellipsoid in $\mu$-space defined by its center $\bar \mu$ and its inverse shape matrix $M$, and for simplicity we presume
that the quadratic inequality right-hand side is $1$). Then $Z$ transforms this ellipsoid into the space of $\lambda$'s as
\[
\{ \lambda \in \mathbb{R}^m: A \lambda = 0, (\lambda - \bar \lambda)^\top (Z M Z^\top) (\lambda - \bar \lambda) \leq 1 \} \ ,
\]
where $\bar \lambda := Z \bar \mu$. We can work with this center $\bar \lambda$ and the ``inverse shape matrix'' $\hat M := Z M Z^\top$.
(Of course, this matrix has rank $p$ and is not invertible.) Evaluating the constraints of $\malt$ at $\mu = \bar \mu$ corresponds
exactly to evaluating the constraints $\lambda \geq 0, \ u^\top \lambda < 0$ at $\lambda = \bar \lambda$. A constraint normal $v$ in $\mu$-space
corresponds to the constraint normal $\hat v$ in $\lambda$-space via the correspondence $v = Z^\top \hat v$. Furthermore, the update of the inverse shape matrix
\[
M_+ = \delta \left( M - \sigma \frac{M v v^\top M} {v^\top M v} \right)
\]
in $\mu$-space corresponds exactly to the update of the ``inverse shape matrix'' in $\lambda$-space given by
\[
\hat M_+ = \delta \left( \hat M - \sigma \frac{\hat M \hat v \hat v^\top \hat M} {\hat v^\top \hat M \hat v} \right) \ ,
\]
where the scalar parameters $\delta$ and $\sigma$ are chosen appropriate to $p$-dimensional space rather than
$m$-dimensional space. The initial ``inverse shape matrix'' is $ZZ^\top$ and can be computed in $O(m^2 \min\{n,p\})$  operations.  However,
updating $\hat M$ at each iteration requires $O(m^2)$ operations. We summarize the above in the following.

\begin{Prop}\label{mission} Using the QR factorization approach outlined above, the standard ellipsoid algorithm can be implemented to solve $\alt$ via the transformed problem $\malt$ starting with the unit ball in either $\nulla$ or $\mu$-space.  Euclidean distances are preserved by the transformations involved, and hence any Euclidean ball of radius $\bar r$ in $\nulla$ corresponds to a Euclidean ball of radius $\bar r$ in $\mu$-space.  Each iteration of the ellipsoid algorithm uses  $O(mp)$ operations in $\mu$-space and $O(m^2)$ operations in $\nulla$. \qed
\end{Prop}

\begin{Rem}\label{missionaccomplished} Because the QR factorization approach preserves Euclidean distances, it follows from  Proposition \ref{mission} and Lemma \ref{hotsunday} that the total number of operations required to compute a solution of $\alt$ using the standard ellipsoid algorithm in $\mu$-space is $O(mp^3 \ln( \frac{m}{\cc \tee}))$.  This bound is shown in the last column of the first row of Table \ref{comp}.\end{Rem}

\subsection{A matrix partition factorization to reduce the operations count of an iteration of the standard ellipsoid algorithm for solving $\alt$}\label{eric2}

This subsection considers a matrix partitioning approach to implement the standard ellipsoid algorithm for solving $\alt$ that reduces the operation counts of an iteration of the ellipsoid algorithm, albeit at the possible expense of worsening the conditioning of the transformed problem.  We will parameterize $\nulla$ in a way that will decrease the per-iteration operations of the standard ellipsoid algorithm for the unbalanced cases in which $n \ll m$ or $p \ll m$.  We can partition $A$ (assuming, for simplicity, that the leading
$n \times n$ submatrix is nonsingular) as $[ A_B, A_N]$. Then the nullspace of $A$ can be represented as
\[
\bigg\{ \lambda = \left( \begin{array}{cc} \lambda_B \\ \lambda_N \end{array} \right) = \left[ \begin{array}{cc} H \\ -I \end{array} \right] \mu : \mu \in \mathbb{R}^p \bigg\} \ ,
\]
where 
\begin{equation}\label{eq_H}
H := A_B^{-1} A_N \ .
\end{equation}
Computing $H$ takes $O(mn^2)$ operations, but needs to be done only once. Here $\mu \in \mathbb{R}^{p}$ again parametrizes the subspace, but now Euclidean distance is not preserved between corresponding pairs of points in $\nulla$ and $\mu$-space (unless $A_N = 0$).  Consider the following transformed version of $\alt$: 
\begin{align*} 
H \mu & \geq 0 \\
 - \mu & \geq 0 \\
 (u_B^{\top} A_B^{-1} A_N - u_N^{\top}) \mu & < 0 \\
 \| \mu \| & \leq 1 \ , 
\end{align*} 
where we have similarly partitioned $u$ into $(u_B; u_N)$, and have added the unit ball constraint due to the positive homogeneity (of degree $1$) of the rest of the system just as we did earlier.  We can apply the ellipsoid method to seek a point satisfying the above system starting with the initial ellipsoid equal to the unit ball centered at the origin.  For convenience, let us write the above transformed system as
\begin{equation*}\begin{array}{ccc}
\mmalt \left\{ \begin{array}{rrl} & -H \mu & \leq 0 \\
& \bar g^\top \mu & < 0 \\
&  \mu & \leq 0 \\
& \| \mu \| & \leq 1 \ , \end{array} \right. \end{array}
\end{equation*} 
where $\bar g = (A_N^\top A_B^{-\top} u_B - u_N)$, and notice that by defining the $ p \times (n+1)$ matrix $G := [- H^\top, \bar g]$, then $G^\top \mu$ is comprised of the left-hand side of the first $(p+1)$ \emph{general} inequalities of $\mmalt$ above. 
Since $\mmalt$ contains $m+1$ linear inequalities in $p$ variables in addition to the unit ball constraint, each iteration would normally require $O(mp)$ operations to
evaluate the constraints at the current center, and $O(p^2)$ operations to apply the inverse of the shape matrix to a constraint vector $v$ of the above system and to update the inverse or a Cholesky factorization of the $p \times p$ shape matrix.
However, because there are only $n+1$ general inequalities (which comprise the matrix $G$), both of these counts can be reduced to 
$O(np)$. This is immediately apparent for evaluating the constraints because evaluating $\mu \leq 0$ requires $O(p)$ operations. 
Below we show how to perform the other tasks in $O(np)$ operations when $n < p$; otherwise, we have $p^2 = O(np)$ already.

Recall from above that our initial ellipsoid is a unit ball; accordingly, the initial shape matrix is the identity matrix. At each iteration of the ellipsoid algorithm, we add a multiple of a rank-one matrix of the form $vv^{\top}$ to the current shape matrix and then positively rescale the shape matrix, where $v$ is a constraint vector from $\mmalt$.  Notice that we can presume that $v$ corresponds to one of the rows of the linear inequalities in $\mmalt$, since if the only violated constraint at the current center is the unit ball constraint $\|\mu\| \le 1$, then the center is indeed a solution of $\alt$ and we are done.  Hence, at every iteration the shape matrix is of the form
\begin{equation} \label{eq_specialB}
B = E + G D G^{\top},
\end{equation}
where $E$ and $D$ are positive semidefinite diagonal matrices of order $p$ and $n+1$, respectively. Note that $E$ is nonsingular, and we will without loss of generality assume that $D$ is as well because we can restrict our attention to the columns of $G$ that correspond to the positive diagonal entries of $D$. Clearly we can apply $B$ to any vector
at a cost of only $O(np)$ operations with this form. By the 
Inverse Matrix Modification Formula (commonly called the Sherman-Morrison-Woodbury formula,
but due to earlier work by Guttman and Duncan --- see Hager \cite{hager}) ,
\[
B^{-1} = E^{-1} - E^{-1} G (D^{-1} + G^{\top} E^{-1} G)^{-1} G^{\top} E^{-1} \ .
\]
Hence, if we have the inverse or a factorization of the inner matrix
\[
J := D^{-1} + G^{\top} E^{-1} G \ , 
\]
we can also apply $B^{-1}$ to any vector in only $O(np)$ operations. It remains to verify that we can update the shape matrix in $O(np)$ operations. At each iteration, we either increase a single entry of $D$ or $E$, and then scale the resulting matrix. If we keep the scalings separate, we have just a change
of a single entry of $D^{-1}$ or $E^{-1}$ in $J$, which leads to a rank-one update. We can therefore update the inverse or
a factorization of $J$ in just $O(n^2)$ operations. (The case when a column is added to $G$, corresponding to a diagonal entry
of $D$ increasing from zero, can also be handled in $O(n^2)$ operations, but we omit the details.)

Thus, all the steps of an iteration of the ellipsoid method applied to $\mmalt$ require just $O(np)$ operations. We summarize the above discussion in the following proposition. 

\begin{Prop}\label{tomcruise} Using the matrix partitioning approach outlined above, the standard ellipsoid algorithm can be implemented to solve $\alt$ via the transformed problem $\mmalt$ starting with the unit ball in $\mu$-space.  Each iteration of the ellipsoid algorithm in $\mu$-space can be implemented using $O(np)$ operations.  Euclidean distances are not preserved by the transformations involved, and hence a Euclidean ball of radius $\bar r$ in $\nulla$ does not necessarily correspond to a Euclidean ball of radius $\bar r$ in $\mu$-space.   \qed
\end{Prop}

\begin{Rem}\label{tomcruiseaccomplished} Because the matrix partitioning approach requires $O(np)$ operations per iteration, the total number of operations required to compute a solution of $\alt$ using the standard ellipsoid algorithm applied to $\mmalt$ is $O(np^3 \ln( \frac{m}{\rho(\tilde A) \tau(\tilde A, u)}))$, where $\hat A := [I, A_B^{-1}A_N]$ is the transformed data matrix and $\tilde A$ is a rescaling of $\hat A$ so that its columns have unit norm.  Thus while this approach reduces the per-iteration operations by the factor $m/n$, the values of the problem condition measures $\rho(\cdot)$ and $\tau(\cdot)$ are changed by the transformation.
\end{Rem}

\subsection{A matrix partition factorization to reduce the operations count of an iteration of the standard ellipsoid algorithm for solving $\poi$}\label{hirob}

Let us now consider the matrix partition factorization of Section \ref{eric2} applied to using the ellipsoid algorithm to solve $\poi$.  Because $\poi$ contains $m$ inequalities in $n$ variables, a straightforward implementation of the ellipsoid algorithm uses $O(mn)$ operations
to evaluate the constraints at the current center, and $O(n^2)$ operations to apply the inverse of the shape matrix to the vector
$v$ of the current violated constraint ($v = a_j$ for some $j \in \{1, \ldots, m\}$), and to update the inverse or a factorization of the shape matrix. Applying the matrix partition approach to $\poi$ and using identical notation as in Section \ref{eric2}, we can write $\poi$ as
\begin{align*} 
& A_B^T x \leq u_B \\ 
& A_N^T x \leq u_N \ ,
\end{align*} 
and hence we can transform $\poi$ into the following system $\poip$ that is defined in terms of the linearly transformed variables $z := A_B^T x$: 
\begin{align*} 
\poip: \ \ \left \{
\begin{aligned} \nonumber
z & \leq u_B \\ 
H^T z & \leq u_N \ ,
\end{aligned} 
\right.
\end{align*} 
where $H$ is given by (\ref{eq_H}) above. Notice that the invertible linear transformation $z := A_B^T x$ is not guaranteed to preserve Euclidean distances. The system $\poip$ is comprised of the $n$ inequalities $z \leq u_B$ and $p$ general inequalities in $n$ variables. It follows immediately that we can evaluate the constraints at the current center in $O(np)$ operations. When  $p < n$, we can apply the inverse of the shape matrix to a constraint vector
and update the inverse or a Cholesky factorization of the shape matrix in $O(np)$ operations by using the techniques described above
for $B$ in (\ref{eq_specialB}). And when $n < p$, we have $n^2 = O(np)$ already. 

It follows from the above discussion that the standard ellipsoid algorithm, OEA-No-Alt, and OEA-MM all require $O(np)$ operations per iteration. To claim a similar result for OEA, we need to show how to update the certificate matrix 
 $\Lambda$ in $O(mp)$ operations per iteration. It is not hard to see that because of the special form of the reformulated constraint matrix, it suffices to store and update the last $p$ rows of $\Lambda$. If infeasibility is detected, the first $n$ rows of $\Lambda e_j$ can be obtained by using the defining equation $A \Lambda = -A$. As a result, we can perform each iteration of OEA in $O(mp)$ operations. We summarize this discussion in the following. 

\begin{Prop}\label{tomcruise2} Using the matrix partitioning approach above, the standard ellipsoid algorithm, OEA, OEA-No-Alt, and OEA-MM can be implemented to solve $\poi$ via the transformed problem $\poip$. Each iteration of the standard ellipsoid algorithm, OEA-No-Alt, and OEA-MM can be implemented using $O(np)$ operations, while each iteration of OEA can be implemented using $O(mp)$ operations. Euclidean distances are not preserved by the transformations involved, and hence a Euclidean ball of radius $\bar r$ in $z$-space does not necessarily correspond to a Euclidean ball of radius $\bar r$ in $x$-space. \end{Prop}

\begin{Rem}\label{tomcruise2accomplished} Similar to Remark \ref{tomcruiseaccomplished}, the total number of operations required to compute a solution of $\poi$ using the standard ellipsoid algorithm, OEA-No-Alt, or OEA-MM applied to $\poip$ is $O(n^3p \ln( \frac{m}{\rho(\tilde A) \tau(\tilde A, u)}))$, and for OEA it is $O(mn^2p \ln( \frac{m}{\rho(\tilde A) \tau(\tilde A, u)}))$.  Here $\hat A := [I, H^\top] = [I, A_B^{-1}A_N]$ is the transformed data matrix and $\tilde A$ is a rescaling of $\hat A$ so that its columns have unit norm.  This approach reduces the per-iteration operations by the factor $m/p$, but the values of the problem condition measures $\rho(\cdot)$ and $\tau(\cdot)$ are changed by the transformation.
\end{Rem}

\section{Connection between $\tee$ and Renegar's distance to ill-posedness $\bar \rho(d)$}\label{jimbo}

The paper \cite{Reneg94} by Renegar develops a rather complete data-perturbation-theoretic condition measure theory for conic optimization using a data-dependent measure $\bar \rho$ that is naturally tied to a variety of geometric, analytic, numerical, and algorithmic properties of conic optimization problems.  We will show below that $\tee \ge \bar \rho(A,u)$, but first we need to 
establish the setting and
 then give a formal definition of $\bar \rho$.

The condition measure $\bar \rho$ is concerned with data-instance-specific conic systems and their state changes 
as the data is perturbed.  Restricting our discussion to the case of linear inequality systems of the form $\poi$, let us define the data $d = (A,u)  \in \mathbb{R}^{n \times m} \times \mathbb{R}^m$ and $\cp_d := \cp_{A,u} := \{ x \in \mathbb{R}^n : A^{\top}x \le u \}$. (We slightly abuse notation in calling the data $d$ in order to be consistent with the notation used in the condition measure theory.) The feasible and infeasible data instances are then defined as ${\cal F} := \{ d  \in  \mathbb{R}^{n \times m} \times \mathbb{R}^m : \cp_d  \ne \emptyset\} $ and ${\cal I} := \{ d  \in  \mathbb{R}^{n \times m} \times \mathbb{R}^m : \cp_d  = \emptyset\} $.  We will define the following norm on the data:
 $\|d\| := \|(A,u)\| := \max\{\|A\|_{1,2}, \|u\|_\infty \}$ where 
 recall that
 the operator norm of a matrix $M$ is 
  $\|M\|_{1,2} = \max_{\|v\|_1 =1} \|Mv\|_2$.  
  The condition measure $\bar \rho(d)$  is 
  then
  defined as:
$$ \bar \rho(d) :=  \left\{ \begin{array}{lr}  \inf_{d + \Delta d \in {\cal I}} \|\Delta d\|   & \mbox{if~} d \in {\cal F} \\ \\ \inf_{d + \Delta d \in {\cal F}} \|\Delta d\|   & \mbox{if~} d \in {\cal I}   \end{array}  \right. \ , $$
which is essentially the size of the smallest data perturbation $\Delta d = (\Delta A, \Delta u)$ for which $\cp_{d+\Delta d}$ changes from nonempty to empty, or {\it vice versa}.  $\bar \rho(d)$ is called the ``distance to ill-posedness'' because the optimal or nearly-optimal perturbed data $d+\Delta d$ lies on the set of ill-posed instances $\partial \cal F =\partial \cal I$.  It is simple to show that $\bar \rho(d)$ and $\tee$ are related as follows:
\begin{equation}\label{jimmy}
\tee \ge \bar \rho(A,u) \ . \end{equation} 
To see this, first consider the case when $d = (A,u) \in \cal F$, and define $\Delta A = 0$ and $\Delta u = (-\tee - \varepsilon) e $, and notice from the definition of $\tee$ that $\cp_{d + \Delta d} = \emptyset$ for all $\varepsilon >0$, whereby $\bar \rho(d) \le \|\Delta d\| = \tee + \varepsilon$, and it then follows that $\bar \rho(d) \le \tee$.  Next consider the case when $d = (A,u) \in \cal I$, and define $\Delta A = 0$ and $\Delta u = \tee e $, and notice from the definition of $\tee$ that $\cp_{d + \Delta d} \ne \emptyset$, whereby $\bar \rho(d) \le \|\Delta d\| = \tee $.

The inequality \eqref{jimmy} is the ``good'' direction for complexity of the ellipsoid method, since the computational complexity shown herein is $O(\ln(1/\tee)) \le O(\ln(1/\bar{\rho}(A,u)))$, which automatically bounds the computational complexity of the ellipsoid method in terms of $\bar \rho(A,u)$.

\section{Iteration Complexity of a Standard Ellipsoid Algorithm for Computing a Solution of $\alt$}\label{calt}

In the case when $\poi$ is infeasible, we derive a bound on the iteration complexity of computing a solution of $\alt$ using the standard ellipsoid algorithm, that depends only on $m$, $\tee$, and $\cc$.  Let $\nulla$ denote the nullspace of $A$ and let $\pa$ denote the $\ell_2$ projection matrix onto $\nulla$, namely $\pa = I - A^T(AA^T)^{-1}A$.  Because $\alt$ is positively homogeneous, we can augment $\alt$ by adding a unit $\ell_2$ ball constraint:
\begin{equation} \nonumber 
\galt:  \ \ \left\{ \begin{array}{rcl}A\lambda & =& 0 \\  
\lambda & \ge& 0 \\
u^{\top}\lambda & <& 0 \\
\|\lambda\| & \le & 1 \ , \end{array} \right.
\end{equation} 
and then solve $\galt$ using the standard ellipsoid algorithm in $\nulla$ (which is a $p := (m-n)$-dimensional subspace of $\mathbb{R}^m$) starting with the unit ball as the starting ellipsoid.  Appendix \ref{impl_Alt} describes how the algorithm can be implemented.  Let $\infeas$ denote the set of solutions of $\galt$.  Recall from \eqref{tau} and \eqref{fridays} the definitions of the condition measure $\tee$ which measures just how infeasible the system $\poi$ is, and $\cc$ which measures how ``bounded'' are the inequalities in $\poi$ in the sense of how much their normals must be perturbed in order to have a non-trivial recession cone.  

Let $B(c,r)$ denote the $\ell_2$ ball in $\mathbb{R}^m$ centered at $c$ with radius $r$.  Critical to bounding the number of iterations of the ellipsoid algorithm when solving $\galt$ is the existence of a ball $B(\lambda^c, r)$ that satisfies $$\lambda^c \in \nulla \ \ \ \ \mathrm{and} \ \ \ \ B(\lambda^c, r)\cap \nulla \subset \infeas \ . $$If this is the case, then the standard ellipsoid algorithm will need at most  $\lceil 2p(p+1)\ln(1/r) \rceil$ iterations to compute a solution of $\infeas$.  It turns out that we can bound the radius $r$ of the largest such ball from above and below in terms of $\tee$ and $\cc$ as the following lemma indicates.

\begin{Lem}\label{hotsunday} In the case when $\poi$ is infeasible, let $r^*$ be the supremum of $r$ for which there exists $\lambda^c$ satisfying
\begin{equation}\label{88} 
\lambda^c \in \nulla \ \ \ \mbox{and} \ \ \ B(\lambda^c, r)\cap \nulla \subset \infeas \ . 
\end{equation}
Then the following holds:
\begin{equation}\label{98}   \frac{\|\pa u\|}{\tee \sqrt{m}} + 1 \ \ \le \ \ \frac{1}{r^*} \ \ \le \ \ \left( \frac{\|\pa u\|}{\tee} + 1\right)\left( \frac{m}{\cc} + \sqrt{m} + 1 \right)   \ . \end{equation}                                                              
\end{Lem} \qed

\noindent For clarity, these bounds can be weakened to:

\begin{equation}\label{89}   \frac{\|\pa u\|}{\tee\sqrt{m}} \  \le \ \frac{1}{r^*} - 1 \ \le \ \frac{(4m+1)\|\pa u\|}{\tee\cc}  \  \end{equation}
(this follows since $\cc \le 1$ and $\tee \le \|\pa u\|$, see below), whereby we see that $\tee$ approximates $r^*$ to within a factor of $m/\cc$.  By combining Lemma \ref{hotsunday} with Proposition \ref{mission}, we obtain the following bound for the standard ellipsoid algorithm for computing a solution of $\infeas$, where $p:=m-n$ is the dimension of $\nulla$.

\begin{Th}\label{warmsunday} In the case when $\poi$ is infeasible, the number of iterations of the standard ellipsoid algorithm applied to solve $\infeas$ starting with the unit ball in $\nulla$ is at most
$$\left\lceil 2p(p+1)\ln\left(\frac{(4m+2)\|\pa u\|}{\tee\cc}  \right)\right\rceil  \ . $$ 
Furthermore, each iteration of the ellipsoid method can be implemented using $O(mp)$ operations.                                             
\end{Th} \qed\medskip

In the case when $\poi$ is infeasible, the linear optimization problem describing $\tee$, together with its dual problem, can be written as:
\begin{equation}\label{82} 
\begin{array}{rrclrrrll}  \tee \ \ = \ \  \displaystyle\min_{x,\tau} & \tau &   &  & = & \displaystyle\max_{\lambda} & -u^T\lambda \\
 \ \ \ \ \ \ \ \mathrm{s.t.} &A^Tx &\le& u + e \tau & &  \ \ \ \ \ \ \ \mathrm{s.t.} &A\lambda &=& 0 \\
  & & & & & &e^T\lambda  & = & 1 \\
 & & & & & &\lambda  &\ge& 0 \ ,
\end{array}
\end{equation}
where we call the left-side and right-side problems above the primal and dual, respectively.  Notice that $(x,\tau) = ((AA^T)^{-1}Au, \|\pa u\|)$ is feasible for the primal problem, and hence $\tee \le \|\pa u\|$.  Thus a natural condition measure for infeasibility is $\|\pa u\|/\tee$ which is (positively) scale invariant and satisfies $\|\pa u\|/\tee \in [1, \infty)$.  Also, because the columns of $A$ have unit norm, it follows that $\|A\|_{1,2} = 1$ and $\cc \le 1$ and hence $1/\cc$ is a natural condition measure that satisfies $1/\cc \in [1, \infty)$.  

If we normalize $u$ for discussion's sake so that $\|\pa u\| = 1$ and we ignore constants, then the left inequality in \eqref{89} states that $\tee\sqrt{m}$ is an upper bound on the largest ball radius of a ball in $\infeas$.  And the right inequality in \eqref{89} states that the largest such ball radius $r^*$ is only guaranteed to be as large as $\tee\cc/m$.  Hence when $m$ is of the same order as $p = m-n$ and $m / \cc \gg 1 / \tee$, then OEA will have a better complexity bound than the bound for the standard ellipsoid algorithm in Theorem \ref{warmsunday}.  Indeed because the complexity bound for OEA in the infeasible case relies only on $\tee$ and $m$, and has no dependence on $\cc$, it can outperform the standard ellipsoid algorithm in terms of its iteration complexity bound.  \medskip

\noindent {\bf Proof of Lemma \ref{hotsunday}:}   To prove the left side of \eqref{98}, let $\lambda, r^*$ satisfy the limiting supremum in \eqref{88}, which therefore satisfy $A\lambda = 0$, $\|\lambda\|+r^* = 1$, and $$Ad=0 \ , \ \|d\| \le 1 \ \implies \ u^T(\lambda + r^* d) \le 0 \ \mathrm{and} \ \lambda + r^* d \ge 0 \ . $$
It therefore follows using $d = \pa u /\|\pa u\|$ that $-u^T\lambda \ge r^* u^T d = r^* u^T\pa u/\|\pa u\| = r^* \|\pa u\|$.  Furthermore, setting $\lambda' := \lambda/e^T\lambda$, it follows that $\lambda'$ is feasible for the dual problem in \eqref{82} and hence $$\tee \ge -u^T\lambda' \ge  \frac{r^* \|\pa u\|}{e^T\lambda} \ge \frac{r^* \|\pa u\|}{\sqrt{m}\|\lambda\|} = \frac{r^* \|\pa u\|}{\sqrt{m}(1-r^*)} \ , $$ and rearranging the terms above yields the left side of \eqref{98}.

To prove the right side of \eqref{98}, we proceed as follows.  Define $\be := \tfrac{1}{m}e$, and for $\rho < \cc$ define the perturbation matrix $\Delta A := \tfrac{\rho}{\|A\be\|}A\be e^T$.  It then follows that $\|\Delta A\|_{1,2} = \rho < \cc$, whereby from the definition of $\cc$ in \eqref{fridays} there does not exist $v \ne 0$ satisfying $[A + \Delta A]^Tv \le 0$.  It then follows from a theorem of the alternative that there exists $\lambda \in \mathbb{R}^m$ satisfying:
\begin{equation}\label{rollingstones}
[A + \Delta A]\lambda = 0 \ , \ \lambda \ge 0 \ , \ \| \lambda\| = 1 \ . 
\end{equation}
Define $\tlambda := \lambda + \tfrac{\rho e^T\lambda}{\|A\be\|}\be $, and it follows from \eqref{rollingstones} that 
\begin{equation}\label{beatles}
A \tlambda = 0 \ , \ \tlambda \ge 0 \ , \ \| \tlambda\| \le 1 + \frac{\rho e^T\lambda}{\sqrt{m}\|A\be\|} \ , \ \tlambda_j \ge \frac{\rho e^T\lambda}{m \|A\be\|} \ \ \mathrm{for} \  j \in [m] \ . 
\end{equation}
Next let $\clambda$ solve the dual problem in \eqref{82}, whereby $\clambda$ satisfies
\begin{equation}\label{who}
A \clambda = 0 \ , \ \clambda \ge 0 \ , \ e^T\clambda =1 \ , \ -u^T\clambda = \tee \ . 
\end{equation}
Now the idea is to take a nonnegative combination of $\clambda$ and $\tlambda$ in a way that guarantees that resulting combination satisfies the inequalities of $\infeas$ as much as possible.  We accomplish this by defining:
\begin{equation}\label{ledzeppelin}
\blambda := \clambda + \alpha \tlambda \ \ \ \mathrm{and} \ \ \ \bar r:= \frac{\alpha \rho e^T\lambda}{m\|A\be\|} \ , \ \ \mathrm{where~~} \alpha := \frac{\tee m \|A\be\|}{\|\pa u\|(m\|\tlambda\|\|A\be\| + \rho e^T\lambda)} \ . 
\end{equation}
Notice that $A\blambda = 0$.  We will now show that $\blambda$ and $\bar r$ also satisfy:
\begin{equation}\label{animals}
Ad=0 \ , \ \|d\| \le 1 \ \implies \  \blambda + \bar r d \ge 0 \ \mathrm{and}  \ u^T(\blambda + \bar r d) \le 0 \ . 
\end{equation} This then implies that upon rescaling $\blambda$ and $\bar r$ by $(\|\blambda\| + \bar r)$ to  
\begin{equation}\label{kinks}
\blambda' := \frac{\blambda}{\|\blambda\| + \bar r} \ \ \ \ \mathrm{and} \ \ \ \ \bar r' := \frac{\bar r}{\|\blambda\| + \bar r}  \ ,
\end{equation} for all sufficiently small positive $\varepsilon$ it holds that 
\begin{equation}\label{doors}\blambda' \in \nulla \ \ \ \ \mathrm{and} \ \ \ \ B(\blambda', \bar r' - \varepsilon) \cap \nulla \subset \infeas \ , \end{equation} and the proof will be completed by then showing a relevant lower bound on the value of $\bar r'$.  Let us therefore now show \eqref{animals}.  It follows from \eqref{beatles}, \eqref{who}, and \eqref{ledzeppelin} that $\blambda_j \ge \alpha \tlambda_j \ge \frac{\alpha \rho e^T\lambda}{m \|A\be\|} = \bar r$, and hence for $d$ satisfying $Ad=0$ and $\|d\| \le 1$ it holds that $\blambda_j  + \bar r d_j \ge \bar r (1-\|d\|) \ge 0$ and hence $\blambda  + \bar r d \ge 0$.  Again considering $d$ satisfying $Ad=0$ and $\|d\| \le 1$, it also follows from \eqref{beatles}, \eqref{who}, and \eqref{ledzeppelin} that $$\begin{array}{rcl}
u^T(\blambda + \bar r d) &=& u^T(\clambda + \alpha \tlambda + \bar r d) \\
&=& -\tee +  \alpha u^T\pa\tlambda + \bar r u^T\pa d \\
&\le& -\tee +  \alpha \|\pa u\|\|\tlambda\| + \bar r \|\pa u\|  \\
&=& -\tee +  \alpha \|\pa u\|\|\tlambda\| + \frac{\alpha \rho e^T\lambda \|\pa u\|}{m\|A\be\|}   \\
&=& -\tee +  \alpha \left(\frac{\|\pa u\|(m\|\tlambda\|\|A\be\| + \rho e^T\lambda)}{ m \|A\be\|}\right) \\
&=& -\tee +  \tee = 0 \ ,  \\
\end{array}
$$thus demonstrating \eqref{animals}.  Applying the rescaling in \eqref{kinks}, it indeed follows that \eqref{doors} holds.  Therefore $r^* \ge \bar r' - \varepsilon$ for all sufficiently small positive $\varepsilon$, and hence $r^* \ge \bar r' $ and 
$$\begin{array}{rcl} \displaystyle\frac{1}{r^*} \le \frac{1}{\bar r'} &=& 1 +  \displaystyle\frac{\|\blambda\|}{\bar r} \medskip \\ 
&\le& 1 +  \displaystyle\frac{\|\clambda\|+ \alpha\|\tlambda\|}{\bar r}  \medskip \\
&\le& 1 + \displaystyle\frac{1}{\bar r} + \displaystyle\frac{\alpha}{\bar r}\left( 1 + \frac{\rho e^T\lambda}{\sqrt{m}\|A\be\|} \right) \medskip \\
&=& 1 + \displaystyle\frac{1}{\bar r} + \displaystyle\frac{\alpha}{\bar r} + \frac{\alpha\rho e^T\lambda}{\bar r \sqrt{m}\|A\be\|} \medskip \\
&=& 1 + \displaystyle\frac{m\|A\be\|}{\alpha \rho e^T\lambda} + \displaystyle\frac{m\|A\be\|}{\rho e^T\lambda}  + \sqrt{m} \medskip \\
&=& 1 + \sqrt{m} +  \displaystyle\frac{m\|A\be\|}{\rho e^T\lambda} + \displaystyle\frac{m\|A\be\|\|\pa u\|}{\tee\rho e^T\lambda}\left( \frac{m\|\tlambda\|\|A\be\| + \rho e^T\lambda}{m\|A\be\|} \right) \medskip \\
&=& 1 + \sqrt{m} +  \displaystyle\frac{m\|A\be\|}{\rho e^T\lambda} + \displaystyle\frac{m\|A\be\|\|\pa u\|}{\tee\rho e^T\lambda}\left(\|\tlambda\|+ \frac{\rho e^T\lambda}{m\|A\be\|} \right) \medskip \\
&=& 1 + \sqrt{m} +  \displaystyle\frac{m\|A\be\|}{\rho e^T\lambda} + \displaystyle\frac{\|\pa u\|}{\tee}\left(\frac{m\|A\be\|}{\rho e^T\lambda}\|\tlambda\|+ 1 \right) \medskip \\
&\le& 1 + \sqrt{m} +  \displaystyle\frac{m\|A\be\|}{\rho e^T\lambda} + \displaystyle\frac{\|\pa u\|}{\tee} + \displaystyle\frac{\|\pa u\|m \|A\be\|}{\tee \rho e^T\lambda}\left( 1 + \frac{\rho e^T\lambda}{\sqrt{m}\|A\be\|} \right)  \medskip \\
&=& 1 + \sqrt{m} +  \displaystyle\frac{m\|A\be\|}{\rho e^T\lambda} + \displaystyle\frac{\|\pa u\|}{\tee}\left( 1 + \sqrt{m} +\frac{m\|A\be\|}{\rho e^T\lambda} \right)  \medskip \\
&=& \left( 1 + \sqrt{m} +  \displaystyle\frac{m\|A\be\|}{\rho e^T\lambda}\right) \left(1 + \displaystyle\frac{\|\pa u\|}{\tee} \right)  \medskip \\
&\le& \left( 1 + \sqrt{m} +  \displaystyle\frac{m}{\rho }\right) \left(1 + \displaystyle\frac{\|\pa u\|}{\tee} \right) \ . 
\end{array}$$
As this inequality holds for all $\rho < \cc$, it also holds for $\rho = \cc$, proving the right-side inequality in \eqref{98}. \qed
\section{Two Update Formulas} 
Propositions \ref{p: d_up} and \ref{p: ell_up} below are straightforward results that follow from the Sherman-Morrison formula and algebraic manipulation. 


\begin{Prop} \label{p: d_up} 
Suppose $d \in \mathbb{R}^m_{++}$ and $\ell \in \mathbb{R}^m$ satisfy $f(d,\ell) = 1$. Let $j \in [m]$, $\delta \in \mathbb{R}_+$, and $\tilde{d} = d + \delta e_j.$ Then, 
\begin{equation} \label{e: u_d_B}
B(\tilde{d})^{-1} = B^{-1} -  \left( \frac{\delta}{1+\delta \gamma_j(d,\ell)^2} \right) B^{-1} a_j a_j^T B^{-1},
\end{equation}
\begin{equation} \label{e: u_d_t}
t(\tilde{d},\ell) = t -  \left( \frac{\delta}{1+\delta \gamma_j(d,\ell)^2} \right) t_j A^T B^{-1} a_j,
\end{equation}
and
\begin{equation} \label{e: u_d_f} 
f(\tilde{d},\ell) = 1 + \delta v_j(\ell)^2 - \left( \frac{\delta}{1+\delta \gamma_j(d,\ell)^2} \right) t_j(d,\ell)^2. 
\end{equation} 
\end{Prop} 
\begin{proof} 
For notational convenience, we suppress the dependence on $d$ and $l$ in the quantities $y(\cdot, \cdot)$, $t(\cdot, \cdot)$, and $f(\cdot, \cdot)$, and we write $y$ for $y(d,\ell)$ and $\tilde y$ for $y(\tilde d, \ell) = y(d + \delta e_j, \ell )$, and similarly for $t$ and $f$; we also write $B$ for $B(d)$ and  $\tilde B$ for $B(\tilde d)$.   Let $\theta := \frac{\delta}{1+\delta \gamma_j(d,\ell)^2}$. From the Sherman-Morrison formula, (\ref{e: u_d_B}) holds, and hence
\begin{align}
\tilde{y} & = \tilde{B}^{-1} A\tilde{D}r \nonumber \\
& = \left(B^{-1} - \theta B^{-1}a_ja_j^{\top}B^{-1} \right) \left(ADr + \delta r_j a_j \right) \nonumber \\ 
& = y + (\delta r_j -\theta a_j^{\top} y - \theta \delta r_j \gamma_j^2)B^{-1}a_j \nonumber \\
& = y + \theta (r_j - a_j^{\top} y ) B^{-1}a_j \nonumber \\ 
& =  y - \theta t_j B^{-1}a_j. \label{e: u_d_y} 
\end{align} 
It follows from (\ref{e: u_d_y}) that $\tilde{t} = A^{\top} \tilde{y}- r = t - \theta t_j A^{\top}B^{-1} a_j$, and thus we have (\ref{e: u_d_t}). Hence, 
\begin{align} 
\tilde{t}^{\top} \tilde{D} \tilde{t} & = (t - \theta t_j A^{\top}B^{-1} a_j)^{\top} (D + \delta e_je_j^{\top}) (t- \theta t_j A^{\top}B^{-1} a_j) \nonumber \\ 
& = t^{\top} D t + \theta^2 t_j^2 \gamma_j^2 + \delta t_j^2 - 2 \delta \theta t_j^2 \gamma_j^2 + \delta \theta^2 t_j^2 \gamma_j^4 \nonumber \\
& = t^{\top} D t+ \theta t_j^2 \ , \label{e: u_d_tDt}
\end{align}
where the second equality follows from $ADt = 0$. Also, 
\begin{equation} \label{e: u_d_vDv} 
v^{\top} \tilde{D}v=  v^{\top}(D + \delta e_j e_j^{\top}) v = v^{\top}Dv + \delta v_j^2 \ .
\end{equation}
From (\ref{e: u_d_tDt}) and (\ref{e: u_d_vDv}),  
\begin{align} 
f(\tilde{d},\ell) & = v^{\top} D v - t^{\top} D t+ \delta v_j^2 - \theta t_j^2 \nonumber \\
& = f(d,\ell) + \delta v_j^2 - \theta t_j^2 \nonumber \\
& = 1 + \delta v_j^2 - \theta t_j^2 \ , \nonumber
\end{align} 
and thus (\ref{e: u_d_f}) holds. 
\end{proof} 

\begin{Prop} \label{p: ell_up} 
Suppose $d \in \mathbb{R}^m_{++}$ and $\ell \in \mathbb{R}^m$. Let $j \in [m]$, $\beta \in \mathbb{R}$, and $\tilde{\ell} = \ell + \beta e_j.$ Then
\begin{align} 
& y(d,\tilde{\ell}) = y(d,\ell) + \tfrac{1}{2} \beta d_j  (B(d))^{-1} a_j, \label{e: u_l_y} \\ 
& t(d,\tilde{\ell}) = t(d,\ell) + \tfrac{1}{2}   \beta d_j A^{\top} (B(d))^{-1} a_j - \tfrac{1}{2} \beta e_j, \label{e: u_l_t} \\ 
& f(d,\tilde{\ell}) = f(d,\ell) +  \beta (t_j(d,\ell) - v_j(\ell))d_j  + \tfrac{1}{4} \beta^2 d_j^2 a_j \adj B(d)^{-1}a_j \ . \label{e: u_l_f}
\end{align} 
\end{Prop}


\begin{proof} 
For notational convenience, we suppress the dependence on $d$ and $l$ in the quantities $y(\cdot, \cdot)$, $t(\cdot, \cdot)$, $f(\cdot, \cdot)$, $r(\cdot, \cdot)$, and $v(\cdot, \cdot)$, and we write $y$ for $y(d,\ell)$ and $\tilde y$ for $y(d,\tilde \ell) = y(d, \ell + \beta e_j)$, and similarly for $t$, $f$, $r$, and $v$; we also write $B$ for $B(d)$.  Because $\tilde{r} = r + (\beta/2)e_j$, 
$$\tilde{y} = B^{-1}AD\tilde{r} = y + \frac{\beta}{2} d_j B^{-1} a_j.$$ 
Thus, equation (\ref{e: u_l_y}) holds. From (\ref{e: u_l_y}), 
$$\tilde{t} = A^{\top} \tilde{y} - \tilde{r} =  t + \frac{\beta}{2} d_j A^{\top} B^{-1} a_j - \frac{\beta}{2} e_j,$$ 
and so equation (\ref{e: u_l_t}) holds. It follows from (\ref{e: u_l_t}) that 
\begin{align} 
& \tilde{t}^{\top} D \tilde{t} \nonumber \\
& = (t + \frac{\beta}{2} d_j A^{\top} B^{-1} a_j - \frac{\beta}{2} e_j)^{\top} D (t + \frac{\beta}{2} d_j A^{\top} B^{-1} a_j - \frac{\beta}{2} e_j) \nonumber \\
& = t^{\top} D t - \beta  t_j d_j + \frac{1}{4} \beta^2 (d_jA^{\top} B^{-1} a_j - e_j)^{\top} D (d_jA^{\top} B^{-1} a_j - e_j)  \nonumber \\ 
& = t^{\top} D t - \beta  t_j d_j + \frac{1}{4} \beta^2 (d_j-d_j^2 a_j \adj B^{-1}a_j ) \label{e: u_l_tDt}, 
\end{align} 
where the second equality is from $ADt = 0$. Because $\tilde{v} = v - (\beta/2) e_j$, 
\begin{equation} \label{e: u_l_vDv}  
\tilde{v}^{\top} D \tilde{v} = v^{\top} D v - \beta v_j d_j  + \frac{1}{4} \beta^2 d_j . 
\end{equation} 
From (\ref{e: u_l_tDt}) and (\ref{e: u_l_vDv}), 
\begin{align} 
f(d,\tilde{\ell}) & =  f(d,\ell) +  \beta d_j (t_j - v_j)  + \frac{1}{4}  \beta^2 d_j^2a_j \adj B^{-1}a_j \ , \nonumber
\end{align} 
and thus (\ref{e: u_l_f}) holds. 
\end{proof}

\section{Proofs of Results} 

\subsection{Proofs for Section \ref{s: updating_certificates}} \label{app-119}

\begin{proof}[Proof of Proposition  \ref{p: cert_const}]
We need to show that $\tilde{\lambda}_i$ and $L_i$ satisfy $\lbi$. First note that $\tilde{\lambda}_i  = \Lambda \hat{\lambda}_i^- +  \hat{\lambda}_i^+ \geq 0$ because $\Lambda \geq 0$, $\hat{\lambda}_i^- \geq 0$, and $ \hat{\lambda}_i^+ \geq 0$. Next observe that 
\begin{equation} \label{e: prop2_s1} 
A \tilde{\lambda}_i =  A\Lambda \hat{\lambda}_i^- + A \hat{\lambda}_i^+ = A(\hat{\lambda}_i^+ -  \hat{\lambda}_i^- ) = A \hat{\lambda}_i \ ,  
\end{equation} 
where the second equality follows from $A \Lambda = -A$. Also note that 
\begin{equation} \label{e: prop2_s2} 
A \hat{\lambda}_i = \gamma_i(d,\ell) A D t(d,\ell) - ADA^{\top} B(d)^{-1} a_i = -a_i \ ,
\end{equation} 
where the second equality follows from $A D t(d,\ell) = 0$ and the definition of $B(d)$. From (\ref{e: prop2_s1}) and (\ref{e: prop2_s2}), it holds that $A \tilde{\lambda}_i = -a_i$, and so it remains to verify that $-\tilde{\lambda}_i^{\top} u \geq L_i$. 

For notational convenience, define 
$$z = y(d,\ell) - \frac{1}{\gamma_i(d,\ell)} B(d)^{-1} a_i \ ,$$
and note that $\hat{\lambda}_i = \gamma_i(d,\ell) D(A^{\top} z - r(\ell))$. Also, for each $j \in [m]$, define 
$$\mu_j = \tfrac{1}{2} \gamma_i(d,\ell) d_j \left[ 2 a_j^{\top} z(a_j^{\top} z - r_j(\ell)) - (a_j^{\top}z - \ell_j) (a_j^{\top} z - u_j) \right] \ ,$$
and observe that if $(\hat{\lambda}_i)_j \neq 0$ (so $d_j \neq 0$ and $a_j^{\top} z \neq r_j(\ell)$), then 
\begin{align}
\frac{\mu_j}{(\hat{\lambda}_i)_j} = a_j^{\top}z - \frac{(a_j^{\top}z - \ell_j)(a_j^{\top}z - u_j)}{2(a_j^{\top}z-r_j(\ell))} & = u_j + \frac{(a_j^{\top}z-u_j)^2}{2(a_j^{\top}z - r_j(\ell))} \label{e: bd1}  \\
&  = \ell_j + \frac{(a_j^{\top}z-\ell_j)^2}{2(a_j^{\top}z - r_j(\ell))} \ . \label{e: bd2} 
\end{align}  
Now if $(\hat{\lambda}_i)_j > 0$, then $\sgn((\hat{\lambda}_i)_j) = \sgn(a_j^{\top}z - r_j(\ell))$, and it follows from multiplying (\ref{e: bd1}) by $(\hat{\lambda}_i)_j$ that 
$$\mu_j = (\hat{\lambda}_i)_j \left[ u_j + \frac{(a_j^{\top}z-u_j)^2}{2(a_j^{\top}z - r_j(\ell))} \right] \geq (\hat{\lambda}_i)_j u_j \ .$$
Similarly if $(\hat{\lambda}_i)_j < 0$, then $\sgn((\hat{\lambda}_i)_j) = \sgn(a_j^{\top}z - r_j(\ell))$, and it follows from multiplying (\ref{e: bd2}) by $(\hat{\lambda}_i)_j$ that 
$$\mu_j = (\hat{\lambda}_i)_j \left[ \ell_j + \frac{(a_j^{\top}z-\ell_j)^2}{2(a_j^{\top}z - r_j(\ell))} \right] \geq (\hat{\lambda}_i)_j \ell_j \ .$$
Finally, if $(\hat{\lambda}_i)_j = 0$, then either $d_j = 0$ in which case $\mu_j = 0$, or $a_j^{\top}z = r_j(\ell)$ in which case 
$$\mu_j = -\tfrac{1}{2} \gamma_i(d,\ell) d_j (a_j^{\top}z - \ell_j)(a_j^{\top}z - u_j) \geq 0 \ .$$
Thus from the above we obtain: 
\begin{align} 
(\hat \lambda_i^+)^{\top} u - (\hat \lambda_i ^-)^{\top} \ell \leq \sum_{j=1}^m \mu_j & =  \hat \lambda_i^{\top} A^{\top} z - \tfrac{1}{2} \gamma_i(d,\ell) \sum_{j=1}^m d_j(a_j^{\top} z - \ell_j)(a_j^{\top}z - u_j) \nonumber \\
& = -a_i^{\top} z - \tfrac{1}{2} \gamma_i(d,\ell) \sum_{j=1}^m d_j(a_j^{\top} z - \ell_j)(a_j^{\top}z - u_j) \nonumber \\
& = -a_i^{\top} z \nonumber  \\ & = \gamma_i(d,\ell) - a_i^{\top} y(d,\ell) \ , \label{e: almost_there} 
\end{align} 
where the second equality follows from (\ref{e: prop2_s2}) and the third equality follows from the fact that $z$ satisfies the inequality defining $E(d,\ell)$ with equality. Rearranging (\ref{e: almost_there}) yields  
$$L_i = a_i^{\top}y(d,\ell) - \gamma_i(d,\ell) \leq (\hat\lambda_i^-)^{\top} \ell - (\hat\lambda_i^+)^{\top} u \leq - (\hat\lambda_i^-)^{\top} \Lambda^{\top} u - (\hat\lambda_i^+)^{\top} u  = -\tilde\lambda_i^{\top} u \ ,$$
where the second inequality follows from $-\Lambda^{\top} u \geq \ell$ and $\hat\lambda_i^- \geq 0$. 
\end{proof} 


\begin{proof}[Proof of Proposition \ref{p: alg_correct_f_leq_0}]
Examining Step \ref{kanab} of Procedure \ref{a: cert_for_f_leq_0}, we see from Remark \ref{p: cert_inf_const-1} that $\bar\lambda_j$ is a type-\1 certificate of infeasibility if $\ell_j > u_j$.   If the procedure does not exit at Step \ref{kanab}, it holds that $\tfrac{1}{2}(u - \ell) = v(\ell) \ge 0$.   And from the discussion of Procedure \ref{a: cert_for_f_leq_0} directly preceding Proposition \ref{p: alg_correct_f_leq_0}, $\bar\lambda_k := \lambda_k + e_k$ is a type-\1 certificate of infeasibility as long as Steps \ref{s: beta} -- \ref{s: cert3} of Procedure \ref{a: cert_for_f_leq_0} can be executed as stipulated.  The only steps that require proof of such viability are Steps \ref{s: beta}, \ref{s: j}, \ref{s: k}, and \ref{s: eps}.  

We first examine Step \ref{s: beta}.  Let $i \in [m]$ be selected, and define $\bar{\ell} := \ell - \beta e_i$; then from Proposition \ref{p: ell_up} it follows that 
$$f(d,\bar{\ell}) := f(d,  \ell - \beta e_i) = f(d,\ell) - d_i(a_i^{\top}y(d,\ell) - u_i)\beta + \tfrac{1}{4} d_i^2a_i^{\top}B(d)^{-1}a_i \beta^2  \ , $$ using \eqref{e: u_l_f} and the fact that $t_i(d,\ell) - v_i(\ell) = a_i\adj y(d,\ell) -u_i$.  As the above expression is a strictly convex quadratic in $\beta$ and $f(d,\ell) \le 0$, there is a positive value of $\beta$ for which $ f(d,  \ell - \beta e_i) = 0$, and in fact using the quadratic formula this value of $\beta$ works out to be:
$$\beta = \frac{2 (a_i^{\top}y(d,\ell)-u_i) + 2\sqrt{ (a_i^{\top}y(d,\ell)-u_i)^2 - f(d,\ell) a_i^{\top} B(d)^{-1} a_i}}{d_i a_i^{\top}B(d)^{-1}a_i} \ .$$
Thus Step \ref{s: beta} is executable.  Let us next consider Step \ref{s: j}.  After Steps \ref{s: beta} and \ref{s: l_dec1} are computed, it holds that $f(d,\ell) = 0$.   We must show that there exists an index $j \in [m]$ such that $a_j^{\top} y(d,\ell) \leq u_j$. Suppose there is no such index; then for all $s \in [m]$ it holds that
$$t_s(d,\ell) = a_s^{\top} y(d,\ell) - r_s(\ell) > u_s - r_s(\ell) = v_s(\ell) \ge 0 \ , $$
where the last inequality follows since the original input lower bounds satisfied $v(\ell) \ge 0$ and the updated value of $\ell$ in Step \ref{s: l_dec1} is less than or equal to the original value, whereby it still holds that $v(\ell) \ge 0$.  It then follows that 
$$f(d,\ell) = v(\ell)^{\top} D v(\ell) - t(d,\ell)^{\top}D t(d,\ell) < 0 \ , $$
which yields a contradiction.  Thus there exists an index $j \in [m]$ such that $a_j^{\top} y(d,\ell) \leq u_j$, whereby Step \ref{s: j} is executable.

To see why Step \ref{s: k} is executable, note that the input to Procedure \ref{a: cert_for_f_leq_0} satisfied $f(d,\ell) \leq 0$ and $A\adj y(d,\ell) \not\le u$ (for the original input value $\ell$) and hence implied that $\mathcal{P} = \emptyset$.  Thus for any $y$ there is a violated inequality of the system $\poi$.

Last of all we show that Step \ref{s: eps} is implementable.  At the start of Step \ref{s: eps} we have $f(d,\ell) = 0$, $a_j^{\top} y(d,\ell) \leq u_j$, and $a_k^{\top} y(d,\ell) > u_k$. From Proposition \ref{p: ell_up} and the fact that $f(d,\ell) = 0$, it follows that 
\begin{align*} 
f(d,\ell - \varepsilon e_j) & = f(d,\ell) -\varepsilon (a_j^{\top} y(d,\ell)- u_j )d_j  + \varepsilon^2 \tfrac{1}{4}d_j^2 a_j^{\top} B(d)^{-1} a_j   \\
& =  \varepsilon (u_j - a_j^{\top} y(d,\ell))d_j  + \varepsilon^2 \tfrac{1}{4}d_j^2 a_j^{\top} B(d)^{-1} a_j \ ,
\end{align*} 
and thus $f(d,\ell - \varepsilon e_j)  > 0$ for all $\varepsilon > 0$.   Accordingly, it is sufficient to show that we can take $\varepsilon >0$ and sufficiently small such that $a_k^{\top} y(d,\ell - \varepsilon e_j) - \gamma_k(d,\ell - \varepsilon e_j) > u_k$.  Let us denote $\bar\ell :=\ell - \varepsilon e_j$ for notational convenience, and $h(\varepsilon):= a_k^{\top}y(d,\ell - \varepsilon e_j) - \gamma_k(d,\ell - \varepsilon e_j) - u_k$.  From Proposition \ref{p: ell_up} and the characterization of slab radii in \eqref{running}, it holds for all $\varepsilon >0$ that
\begin{small}
\begin{equation*}\begin{array}{rcl} 
 h(\varepsilon) & =  &a_k^{\top} y(d,\ell) - u_k - \varepsilon \tfrac{1}{2} d_j a_k^{\top} B(d)^{-1} a_j \\
& &- \left( \varepsilon (u_j - a_j^{\top} y(d,\ell))d_j  + \varepsilon^2 \tfrac{1}{4}d_j^2 a_j^{\top} B(d)^{-1} a_j  \right)^{1/2} (a_k^{\top} B^{-1}(d) a_k )^{1/2} \\ \\
& = &\delta - \varepsilon \tfrac{1}{2} d_j a_k^{\top} B(d)^{-1} a_j - \left( \varepsilon (u_j - a_j^{\top} y(d,\ell))d_j  + \varepsilon^2 \tfrac{1}{4}d_j^2 a_j^{\top} B(d)^{-1} a_j  \right)^{1/2} (a_k^{\top} B^{-1}(d) a_k )^{1/2} \ ,
\end{array}\end{equation*}
\end{small}

\noindent where $\delta := a_k^{\top} y(d,\ell) - u_k > 0$.  Now notice that $h(0)= \delta >0$ and by continuity it holds that $h(\varepsilon) > 0$ for all $\varepsilon >0$ and sufficiently small.  Thus Step \ref{s: eps} is implementable.  Furthermore, the equation $h(\varepsilon) = \delta/2$ can be rearranged so that squaring both sides yields a quadratic in $\varepsilon$, and so Step \ref{s: eps} can be implemented using the mechanics of the quadratic formula. \end{proof}

\subsection{Proofs for Section \ref{s: update} } \label{biden}
For notational convenience we define:
\begin{align*}  
& \beta^{(1)} :=  \frac{-2(t_j(d,\ell)-v_j(\ell))}{d_j\gamma_j(d,\ell)^2} \ , \\ 
& \beta^{(2)} := \frac{2(2v_j(\ell^{(1)}) - \gamma_j(d^{(1)},\ell^{(1)}))}{(m-1) d_j^{(1)} \gamma_j(d^{(1)},\ell^{(1)})^2 + 2} \ .
\end{align*} 
Note that $\ell^{(1)} = \ell + \beta^{(1)} e_j$ and $\ell^{(2)} = \ell^{(1)} + \beta^{(2)} e_j$. 


\begin{proof}[Proof of Lemma \ref{l: ss}] 
Observe that 
\begin{equation} \label{e: u1} 
\ell_j - \ell_j^{(1)} = \frac{2(t_j(d,\ell) - v_j(\ell))}{d_j \gamma_j(d,\ell)^2} = \frac{2(a_j^{\top} y(d,\ell) - u_j)}{d_j \gamma_j(d,\ell)^2} > 0 \ ,
\end{equation}
and so $\ell_j^{(1)} < \ell_j$ and hence $\ell^{(1)} \le \ell$, whereby $\ell^{(1)}$ is a lower bound for $\poi$ with certificate matrix $\Lambda$ since $\Lambda$ is a certificate matrix for $\ell$ and $\ell^{(1)} \le \ell$.  From (\ref{e: u_l_y}) with $\beta = \beta^{(1)}$ we have:
\begin{align} 
a_j^{\top} y(d, \ell^{(1)})  & = a_j^{\top} y(d,\ell) + \tfrac{1}{2} \beta^{(1)} d_j \gamma_j(d,\ell)^2 \nonumber \\
& = a_j^{\top} y(d,\ell) - (t_j(d,\ell) - v_j(\ell)) \nonumber \\ 
& = a_j^{\top} y(d,\ell) - (a_j^{\top}y(d,\ell) - r_j(\ell) - v_j(\ell)) \nonumber \\ 
& = u_j \ , \nonumber 
\end{align}  
which shows \eqref{e: shift}.  And from (\ref{e: u_l_f}) with $\beta = \beta^{(1)}$ we have:
\begin{align} 
f(d,\ell^{(1)}) & = 1 + \beta^{(1)} (t_j(d,\ell) - v_j(\ell))d_j + \tfrac{1}{4} (\beta^{(1)})^2 d_j^2 \gamma_j(d,\ell)^2 \nonumber \\
& =  1 - \left( \frac{t_j(d,\ell)-v_j(\ell)}{\gamma_j(d,\ell)} \right)^2 \nonumber \\ 
& =  1 - \left( \frac{a_j^{\top} y(d,\ell) - u_j}{\gamma_j(d,\ell)} \right)^2 \ , \label{e: u2}  
\end{align} 
which demonstrates the equality in \eqref{e: shrink}.  The inequality in \eqref{e: shrink} follows since $j$ is the index of a violated constraint, hence  $a_j^{\top} y(d,\ell) > u_j$.\end{proof} 


\begin{proof}[Proof of Lemma \ref{l: construct}]  We first prove item (a).  First suppose that $\beta^{(2)} \leq 0$.  Then $$\ell_j^{(2)} = \ell_j^{(1)} + \beta^{(2)} \leq \ell_j^{(1)} < \ell_j \le  \max \left \{ \ell_j, L_j   \right \}  , $$
where the strict inequality uses \eqref{e: u1}.  Next suppose that $\beta^{(2)} > 0$.  Then $\beta^{(2)} \leq 2v_j(\ell^{(1)}) - \gamma_j(d^{(1)},\ell^{(1)}),$ and therefore
\begin{align} 
\ell_j^{(2)} & \leq \ell_j^{(1)} +2v_j(\ell^{(1)}) - \gamma_j(d^{(1)},\ell^{(1)}) \nonumber \\
& = u_j -  2v_j(\ell^{(1)})  + 2v_j(\ell^{(1)}) - \gamma_j(d^{(1)},\ell^{(1)}) \nonumber \\
& = a_j^{\top} y(d,\ell^{(1)}) -\gamma_j(d^{(1)},\ell^{(1)}) \ , \label{e: u4} 
\end{align} 
where the last equality follows from \eqref{e: shift}.  Also note that 
\begin{equation} \label{e: utemp} 
\gamma_j(d,\ell^{(1)}) = f(d,\ell^{(1)})^{\frac{1}{2}} (a_j^{\top} B(d)^{-1} a_j)^{\frac{1}{2}} = f(d,\ell^{(1)})^{\frac{1}{2}}   \gamma_j(d,\ell). 
\end{equation} 
From the invariance of $\gamma_j(\cdot, \ell^{(1)})$ under positive scaling of the first argument, 
\begin{align} 
& \left(a_j^{\top} y(d,\ell) - \gamma_j(d,\ell) \right) - \left(a_j^{\top} y(d, \ell^{(1)}) - \gamma_j(d^{(1)}, \ell^{(1)}) \right) \nonumber \\ 
& = \gamma_j(d, \ell^{(1)}) - \gamma_j(d,\ell) +a_j^{\top} y(d,\ell) - a_j^{\top} y(d, \ell^{(1)}) \nonumber \\ 
& = \left( f(d,\ell^{(1)})^{\frac{1}{2}} - 1 \right) \gamma_j(d,\ell) + a_j^{\top} y(d,\ell) - u_j \nonumber \\ 
& = \left( \sqrt{1 - \left( \frac{a_j^{\top} y(d,\ell) - u_j}{\gamma_j(d,\ell)}  \right)^2  }- 1 \right)   \gamma_j(d,\ell) + a_j^{\top} y(d,\ell) - u_j \nonumber \\ 
& \geq  0,  \nonumber
\end{align} 
where the second equality follows from (\ref{e: shift}) and (\ref{e: utemp}), the third equality uses (\ref{e: u2}), and the inequality from the fact that $\sqrt{1-x^2} \geq 1-x$ for any scalar $0 \leq x \leq 1$.  Therefore
\begin{equation} \label{e: u5} 
 a_j^{\top} y(d ,\ell^{(1)}) -\gamma_j(d^{(1)},\ell^{(1)}) \leq a_j^{\top} y(d,\ell) - \gamma_j(d,\ell) \ , 
\end{equation} 
and (\ref{e: u4}) and (\ref{e: u5}) combine to yield  
$$\ell_j^{(2)} \le  a_j^{\top} y(d,\ell) - \gamma_j(d,\ell) \le  \max \left \{ \ell_j, L_j \right \} \ , $$
which completes the proof of (a).

Next, (b) is immediate since $f(d^{(1)},\ell^{(1)}) = 1$ from (\ref{e: update2}).

Finally, we prove (c).  We need to prove that 
\begin{equation} \label{e: fd2l2}
f(d^{(2)},\ell^{(2)}) = \frac{m^2}{m^2 - 1} = 1 + \frac{1}{m^2 - 1} \ .
\end{equation}
Note that
\begin{align}
f(d^{(2)},\ell^{(2)})  & = f(d^{(1)},\ell^{(1)}) + [ f(d^{(2)},\ell^{(1)}) - f(d^{(1)},\ell^{(1)}) ] + [ f(d^{(2)},\ell^{(2)}) - f(d^{(2)},\ell^{(1)}) ] \nonumber \\
  & = 1 + [ f(d^{(2)},\ell^{(1)}) - 1 ] + [ f(d^{(2)},\ell^{(2)}) - f(d^{(2)},\ell^{(1)}) ] \ . \label{e: 3terms}
\end{align}
We now proceed to evaluate the two terms in brackets.

 For notational convenience let 
\begin{align*} 
& y^{(1)} = y(d^{(1)},\ell^{(1)}) \ , \\
& t^{(1)} =  t(d^{(1)},\ell^{(1)}) \ , \\
& \gamma = \gamma_j(d^{(1)},\ell^{(1)}) \ , \\
& v_j^{(1)} = v_j(\ell^{(1)}) \ , \\
& r_j^{(1)} = r_j(\ell^{(1)}) \ , \\ 
& B = B(d^{(1)}) \ .
\end{align*} 
Then from Proposition \ref{p: d_up} with $\delta = 2/[(m-1)\gamma^2]$ and $\theta = \delta / (1 + \delta \gamma^2)$, we have $B(d^{(2)})^{-1} = B^{-1} - \theta B^{-1} a_j a_j^T B^{-1}$, and so
\[
a_j^{\top} B(d^{(2)})^{-1} a_j = \gamma^2 - \theta \gamma^4 = \frac{1}{1 + \delta \gamma^2} \gamma^2 = \frac{m-1}{m+1} \gamma^2 \ .
\]
Also, $t(d^{(2)},\ell^{(1)}) = t^{(1)} - \theta t_j ^{(1)}A^T B^{-1} a_j$, and so
\[
t_j(d^{(2)},\ell^{(1)}) = t_j^{(1)} (1 - \theta \gamma^2) = \frac{m-1}{m+1} t_j^{(1)} \ .
\]
Lastly,
\[
f(d^{(2)},\ell^{(1)}) - 1 = \frac{2}{(m-1)\gamma^2} (v_j^{(1)})^2 - \frac{2}{(m+1)\gamma^2} (t_j^{(1)})^2 \ .
\]
From the invariance of $y(\cdot,\ell^{(1)})$ under positive scaling and (\ref{e: shift}), it holds that 
\begin{equation} \label{e: utemp3} 
t_j^{(1)} = a_j^{\top} y^{(1)} - r_j^{(1)}= a_j^{\top} y(d,\ell^{(1)})- r_j^{(1)} = u_j - r_j^{(1)} = v_j^{(1)} \ , 
\end{equation} 
and so
\begin{equation}\label{e: diff1}
f(d^{(2)},\ell^{(1)}) - 1 = \frac{4}{(m^2 - 1)\gamma^2}  (v_j^{(1)})^2 \ .
\end{equation}

Next, from Proposition \ref{p: ell_up}, we have
\begin{equation} \label{e: second_diff}
f(d^{(2)},\ell^{(2)}) - f(d^{(2)},\ell^{(1)})  = \beta^{(2)} (t_j(d^{(2)},\ell^{(1)}) - v_j^{(1)}) d_j^{(2)} + \frac{1}{4} (\beta^{(2)})^2 (d_j^{(2)})^2 a_j^T (B(d^{(2)}))^{-1} a_j \ .
\end{equation}
Note that from the definition of $d_j^{(2)}$ and $\beta^{(2)}$, it follows that 
\[
\beta^{(2)} d_j^{(2)} = \frac{ 2 (2 v_j^{(1)} - \gamma) } { (m-1) \gamma^2 } \ .
\]
We can now evaluate the terms in (\ref{e: second_diff}).

Since $t_j(d^{(2)},\ell^{(1)}) = \frac{m-1}{m+1} t_j^{(1)} = \frac{m-1}{m+1} v_j^{(1)}$, the first term on the right-hand side is 
\[
\beta^{(2)} d_j^{(2)} \left( \frac{m-1}{m+1} v_j^{(1)} - v_j^{(1)} \right)  =  \frac{ 2 (2 v_j^{(1)} - \gamma) } { (m-1) \gamma^2 }  \cdot \frac{ -2 v_j^{(1)} } {m+1} =
  \frac{ -4 (2 (v_j^{(1)})^2 - \gamma v_j^{(1)} )} {(m^2 - 1) \gamma^2} .
  \]
  The second term on the right-hand side is equal to 
  \[
  \frac{1}{4} (\beta^{(2)} d_j^{(2)} )^2 \frac{m-1}{m+1} \gamma^2 = \frac{ ( 2 v_j^{(1)} - \gamma)^2 } {(m^2 - 1) \gamma^2}.
  \]
  Combining these terms and substituting them in (\ref{e: second_diff}) gives
  \[
f(d^{(2)},\ell^{(2)}) - f(d^{(2)},\ell^{(1)})  =   - \frac{4}{(m^2 - 1) \gamma^2}  (v_j^{(1)})^2 + \frac{1}{m^2 - 1} ,
\]
which with (\ref{e: diff1}) and (\ref{e: 3terms}) yields (\ref{e: fd2l2}).
\end{proof} 

\begin{proof}[Proof of Theorem \ref{t: ssc}] 
Define $\alpha(d,\ell) := \frac{m^2-1}{m^2} \frac{1}{f(d, \ell^{(1)}) }$. Note that $\alpha(d,\ell) > \frac{m^2-1}{m^2}$ because $0 < f(d,\ell^{(1)}) < 1$
using the hypothesis of the theorem and Lemma \ref{l: ss}. From the invariance of $\gamma_j(\cdot, \ell^{(1)})$ under positive scaling of the first argument, it holds that:
\begin{equation} \label{e: ssc1} 
\gamma_j(d^{(1)},\ell^{(1)})^2 = \gamma_j(d, \ell^{(1)})^2 = f(d,\ell^{(1)}) \gamma_j(d,\ell)^2 \ .
\end{equation} 
The result now follows from (\ref{e: update2}) and  Lemma \ref{l: construct} (c).
\end{proof} 


\subsection{Proofs for Section \ref{s: alg} } \label{finnigan}

\begin{proof}[Proof of Proposition \ref{l: 1}]  
For any $x \in E(d,\ell)$ it holds that
$$
d_i(x-y(d,\ell)) a_ia_i^{\top} (x-y(d,\ell))   \leq (x - y(d,\ell) )^{\top} ADA^{\top} (x - y(d,\ell)) \leq f(d,\ell) \ , $$
and therefore $\lvert a_i^{\top}x - a_i^{\top} y(d,\ell) \rvert \leq \left( \frac{f(d,\ell)}{d_i} \right)^{\frac{1}{2}}$ for all $x \in E(d,\ell)$. The result then follows from the definition of $\gamma_i(d,\ell)$.
\end{proof} 

\begin{proof}[Proof of Lemma \ref{t: main_t}]

We first show that 
\begin{equation} \label{e: term1} 
\mu_j(\tilde{d}) \leq \frac{m}{m+1} \mu_j(d) \ .
\end{equation} 
Note that $\mu_j(d) =  \sqrt{\frac{f(d,\ell)}{d_j}}$ because $ \sqrt{\frac{f(d,\ell)}{d_j}} \geq \tau(A,u) \geq \frac{m}{m+1} \tau(A,u)$. From Proposition \ref{l: 1} and $\alpha \geq \frac{m^2-1}{m^2}$ it follows that:
$$\frac{\tilde d_j}{f(\tilde d,\ell)}= \alpha \left(\frac{d_j}{f(d,\ell)} + \frac{2}{m-1} \frac{1}{\gamma_j(d,\ell)^2} \right) \geq \frac{m^2-1}{m^2} \left(1 + \frac{2}{m-1} \right) \frac{d_j}{f(d,\ell)} = \left( \frac{m+1}{m} \right)^2 \frac{d_j}{f(d,\ell)} \ , $$ 
and therefore
\begin{align*} 
\mu_j(\tilde{d}) & = \max \left \{ \sqrt{\frac{f(\tilde d,\ell)}{\tilde d_j}}, \frac{m}{m+1}  \tau(A,u) \right \} \\
& \leq  \max \left \{ \frac{m}{m+1}\sqrt{\frac{f(d,\ell)}{d_j}}, \frac{m}{m+1}  \tau(A,u) \right \} = \frac{m}{m+1} \mu_j(d) \ , 
\end{align*} 
where the last equality follows from $\sqrt{\frac{f(d,\ell)}{d_j}} \geq \tau(A,u)$ and $\mu_j(d) = \sqrt{\frac{f(d,\ell)}{d_j}}$. 

Next we show that for all $i \in [m]$, $i \neq j$, it holds that:
\begin{equation} \label{e: term2} 
\mu_i(\tilde{d}) \leq \left( \frac{m^2}{m^2-1} \right)^{\frac{1}{2}} \mu_i(d) \ .
\end{equation} 
 Note that $\tilde{d}_i / f(\tilde{d},\ell) = \alpha d_i / f(d,l)$, from which
 \[
 \sqrt{\frac{f(\tilde{d},\ell)}{\tilde{d}_i}} = \sqrt{\frac{1}{\alpha}} \sqrt{\frac{f(d,l)}{d_i}} \leq  \left(\frac{m^2}{m^2-1}\right)^\frac{1}{2}\sqrt{ \frac{f(d,l)}{d_i}} . 
 \]
 Thus if $\mu_i(\tilde{d}) =  \frac{m}{m+1} \tau(A,u)$, we have
 \[
  \mu_i(\tilde{d})  = \frac{m}{m+1} \tau(A,u) \leq \left( \frac{m^2}{m^2-1} \right)^{\frac{1}{2}}  \frac{m}{m+1} \tau(A,u)
      \leq \left( \frac{m^2}{m^2-1} \right)^{\frac{1}{2}} \mu_i(d),
      \]
while if $\mu_i(\tilde{d}) =  \sqrt{\frac{f(\tilde{d},\ell)}{\tilde{d}_i}} $, we have
\[
\mu_i(\tilde{d}) =  \sqrt{\frac{f(\tilde{d},\ell)}{\tilde{d}_i}} \leq  \left(\frac{m^2}{m^2-1}\right)^\frac{1}{2}\sqrt{ \frac{f(d,l)}{d_i}} \leq 
\left( \frac{m^2}{m^2-1} \right)^{\frac{1}{2}} \mu_i(d).
\]
Together these establish (\ref{e: term2}).
 
  Thus from (\ref{e: term1}) and (\ref{e: term2}) we obtain:
\begin{align*} 
\phi(\tilde d, \ell) & \leq \phi(d, \ell)\left( \frac{m}{m+1} \right) \left( \frac{m^2}{m^2 -1} \right)^{(m-1)/2} \\
& = \phi(d, \ell)\left( 1 - \frac{1}{m+1} \right) \left(1 + \frac{1}{m^2-1} \right)^{(m-1)/2} \\
& \leq \phi(d, \ell) e^{-\frac{1}{(m+1)}} \left(e^{-\frac{1}{(m^2-1)}} \right)^{(m-1)/2} \\ 
& = \phi(d, \ell)e^{-\frac{1}{2(m+1)}} \ , 
\end{align*} 
where the second inequality follows from the fact that $1 + x \leq e^x$ for any scalar $x$.  
\end{proof}

\begin{proof}[Proof of Lemma \ref{l: volume_dec}] 
For notational convenience let us assume that $d$ and $\tilde d$ have been rescaled so that $f(d,\ell)=1$ and $f(\tilde d, \tilde \ell)=1$.  Then the volume ratio $E(\tilde{d},\tilde{\ell})$ and $E(d,\ell)$ is:
\begin{align*} 
& \frac{\vol E(\tilde{d}, \tilde{\ell})}{\vol E(d,\ell)} = \left( \frac{\det (ADA^{\top})}{ \alpha^n \det \left(ADA^{\top} + \frac{2}{m-1} \frac{1}{\gamma_j(d,\ell)^2} a_j a_j^{\top}\right)} \right)^{\frac{1}{2}} = \left(\frac{m-1}{m+1} \right)^\frac{1}{2} \left(\frac{1}{\alpha} \right)^{\frac{n}{2}}  \\ 
& \leq  \left(\frac{m-1}{m+1} \right)^\frac{1}{2} \left(\frac{m^2}{m^2-1} \right)^{\frac{n}{2}} \leq  \left(\frac{m-1}{m+1} \right)^\frac{1}{2} \left(\frac{m^2}{m^2-1} \right)^{\frac{m}{2}} \\ 
& =  \left( \frac{m}{m+1} \right) \left( \frac{m^2}{m^2 -1} \right)^{(m-1)/2} \leq e^{-\frac{1}{2(m+1)}}, 
\end{align*} 
where the second equality uses the matrix determinant lemma, and the last inequality follows from the fact that $1 + x \leq e^x$ for any scalar $x$. 
\end{proof} 

\begin{Prop}\label{evie} $\det(AA\adj) \ge \cc^{2n}$.
\end{Prop}

\begin{proof}  If $\cc =0$ the result is clearly true, so suppose for the rest of the proof that $\cc >0$.  We claim that \begin{equation}\label{quiet} \mbox{for~each~} v \ \mbox{satisfying~} \|v\|=1 \ \mbox{there~exists~} i \in [m] \ \mbox{satisfying} \ a_i\adj v \ge \cc \ . \end{equation}
If the claim is true, then for any $v$ with $\|v\|=1$ it holds that $v\adj AA\adj v = \sum_{j=1}^m (a_j\adj v)^2 \ge   (a_i\adj v)^2 \ge \cc^2$, and so the smallest eigenvalue of $AA\adj$ is at least $\cc^2$, whereby $\det(AA\adj) \ge \cc^{2n}$, which proves the result.

We now prove \eqref{quiet} by contradiction.  Suppose that the claim is false.  Then there exists $\bv$ with $\|\bv\| =1$ and 
$0 < \alpha < 1$ 
 satisfying $A\adj \bv \le \alpha \cc e$.  Define $\Delta A := -\alpha \cc \bv e\adj$ and note that $$(A + \Delta A)\adj (-\bv) = -A\adj \bv + 
 \alpha \rho(A) \bv\adj \bv e 
 \ge -\alpha \cc e + \alpha \cc e = 0 \ , $$ and so from the definition of $\cc$ in \eqref{fridays} it must hold that $\|\Delta A\|_{1,2} \ge \cc$.  However, it is simple to verify that $\|\Delta A\|_{1,2} =\|-\alpha \cc \bv e\adj \|_{1,2} = \alpha\cc < \cc$ which provides the desired contradiction, 
 establishing \eqref{quiet}.
  \end{proof}

\subsection{Proofs for Section \ref{covid3} } \label{vistalives}
\begin{proof}[Proof of Proposition \ref{covid2}] 

We will establish a more general result that will imply the correctness of Procedure \ref{postc} as a particular instance.  Recall from \eqref{covid1} that the certificate matrices satisfy the recursion $\Lambda^{(i)} =  \Lambda^{(i-1)}M_{(i)} + B_{(i)}$ for $i=1, \ldots, k$.  Suppose we are given vectors $w^k$ and $z^k$, and we wish to compute $\bar v := \Lambda^{(k)} w^k + z^k$. We claim that the following procedure accomplishes this task:

\medskip 

\begin{algorithmic}[1] 
\For{$i = k:1$}
\State $w^{i-1} \leftarrow M_{(i)} w^i$
\State $z^{i-1} \leftarrow B_{(i)} w^i + z^i$
\EndFor
\State Return $\bar v := \Lambda^{(0)}w^0 + z^0$  
\end{algorithmic}


\medskip

It is straightforward to prove using induction that the returned value $\bar v$ satisfies $\bar v = \Lambda^{(k)} w^k + z^k$. Now notice that Procedure \ref{postc} is simply an instantiation of the above procedure with $w^k := e_{j_k}$, $z^k := e_{j_k}$, and the formulas for $M_{(i)}$ and $B_{(i)}$ in Section \ref{covid3}.
\end{proof}

\bibliographystyle{amsplain}
\bibliography{GF-papers-nips-better
}

\end{document}